\documentclass[a4paper,11pt,reqno]{article}
\usepackage{amsmath}
\usepackage{amsfonts}
\usepackage{graphicx}
\usepackage{color}

\newtheorem{thh}{Theorem}[section]
\newtheorem{cor}[thh]{Corollary}
\newtheorem{lem}[thh]{Lemma}
\newtheorem{prop}[thh]{Proposition}

\def\sep{\;:\;}
\def\proof#1. {\par
                      \ifdim\lastskip<15pt
                      \removelastskip\penalty-200
                      \vskip15pt plus3pt minus3pt
                      \fi
                       {\def\a{#1}
                       \ifx\a\empty
                       {\noindent\bf Proof.}
                       \else
                       {\noindent\bf Proof of #1.}
                       \fi}\enspace}
\def\restr#1{\,\vrule\,\lower1.75ex\hbox{{\scriptsize $ #1$}}}

\def\endproof{\hfill\hspace{-6pt}\rule[-14pt]{6pt}{6pt}
\vskip22pt plus3pt minus 3pt}

\def\be{\begin{equation}}
\def\ee{\end{equation}}
\def\bea{\begin{eqnarray}}
\def\eea{\end{eqnarray}}
\def\bean{\begin{eqnarray*}}
\def\eean{\end{eqnarray*}}

\def\a{\alpha}
\def\b{\beta}
\def\d{\delta}
\def\D{\Delta}
\def\e{\varepsilon}

\def\g{\gamma}
\def\G{\Gamma}
\def\i{\infty}
\def\k{\kappa}
\def\l{\lambda}

\def\o{\omega}
\def\O{\Omega}
\def\r{\rho}
\def\s{\sigma}

\def\t{\theta}

\def\z{\zeta}
\def\c{{\rm cap}}

\def\setm{\setminus}

\def\ov{\overline}

\def\c{{\rm cap}}

\font\tenopen = cmbx10
\font\sevenopen = cmbx7
\font\fiveopen = cmbx5
\newfam\openfam

\textfont\openfam = \tenopen
\scriptfont\openfam = \sevenopen
\scriptscriptfont\openfam = \fiveopen

\def\C{{\mathbb{C}}}

\newcommand{\sect}[1]{\section{#1}\setcounter{equation}{0}}

\title{Christoffel functions with power type weights
\footnote{
words: Christoffel functions, asymptotics, power type
weights, Jordan curves and arcs, Bessel functions, fast
decreasing polynomials, equilibrium measures, Green's functions}
}
\author{Tivadar Danka\thanks{Supported by ERC Advanced Grant No. 267055}
\and Vilmos Totik\thanks{Supported by NSF DMS-1265375}}
\begin{document}
\maketitle
\begin{abstract}
Precise asymptotics for Christoffel functions are established
for power type weights on unions of Jordan curves and arcs.
The asymptotics involve the equilibrium measure of the support
of the measure. The result at the endpoints of arc components
is obtained from the corresponding asymptotics for internal points
with respect to a different power weight. On curve components
the asymptotic formula is proved via a sharp form of Hilbert's lemniscate
theorem while taking polynomial inverse images.
The situation is completely different on
the arc components, where the local asymptotics is obtained
via a discretization of the equilibrium measure
with respect to the zeros of an associated Bessel function.
The proofs are potential theoretical, and fast decreasing
polynomials play an essential role in them.
\end{abstract}
\tableofcontents
\sect{Introduction}
Christoffel functions have been the subject of many papers,
see e.g. \cite{Nevai}, \cite{Nevai1}, \cite{SimonCD}, and the extended
reference lists there. They are
intimately connected with orthogonal polynomials, reproducing
kernels, spectral properties of Jacobi matrices, convergence of
orthogonal expansion and even to random matrices,
see \cite{Grenander}, \cite{Nevai1} and \cite{SimonCD}
for their various connections and applications. The possible
applications are growing, for example  recently
a new domain recovery technique has been devised that
use the asymptotic behavior of
 Christoffel functions, see \cite{Putinar}; and
in the last 4-5 years
several important methods for proving universality in random matrix theory
were based on them, see \cite{Avila}, \cite{Lubinsky0}, \cite{Lubinsky1}
and \cite{Lubinsky2}. The aim of the present paper
is to complete, to a certain extent, the investigations
concerning their asymptotic behavior on Jordan curves and arcs.

Let $\mu$ be a finite Borel measure  on the plane
such that its support is compact  and
consists of infinitely many points.
The Christoffel functions
associated with $\mu$ are defined
as
\be \l_n(\mu,z_0)=\inf_{P_n(z_0)=1}\int|P_n|^2d\mu,\label{infl}\ee
where the infimum is taken for all polynomials
of degree at most $n$ that take the value 1 at $z$.
If $p_k(z)=p_k(\mu,z)$ denote the orthonormal polynomials
with respect to $\mu$, i.e.
\[\int p_n\ov{p_m}d\mu=\d_{n,m},\]
then $\l_n$ can be expressed as
\[\l_n^{-1}(\mu,z)=\sum_{k=0}^n|p_k(z)|^2.\]
In other words, $\l^{-1}(\mu,z)$ is the diagonal of the
reproducing kernel
\[K_n(z,w)=\sum_{k=0}^np_k(z)\ov{p_k(w)}\]
which makes it an essential tool in many problems.
It is easy to see that, with this reproducing kernel, the infimum in (\ref{infl})
is attained (only) for
\[P_n(z)=\frac{K_n(z,z_0)}{K_n(z_0,z_0)},\]
see e.g. \cite[Theorem 3.1.3]{Szego}).

The earliest asymptotics for Christoffel functions for measures
on the unit circle or on $[-1,1]$ go back to Szeg\H{o}, see
\cite[Th. I', p. 461]{Szegocollected}.
He gave their behavior outside the support of the measure,
and for some special cases he also found their
behavior at points of $(-1,1)$. The first
result for a Jordan arc (a circular arc) was given
in \cite{Golinskii}.
By now the asymptotic behavior of Christoffel functions for measures defined on
unions of Jordan curves and arcs $\G$ is well understood: under certain assumptions
 we have
for points $z\in \G$ that are different from the endpoints of
the arc components
of $\G$
\be \lim_{n\to\i}\l_n(\mu,z_0)=\frac{w(z_0)}{\o_\G(z_0)},\label{00}\ee
where $w$ is the density of $\mu$ with respect to the arc measure $s_\G$ on
$\G$, and $\o_\G$ is the density of the equilibrium measure
(see below) with respect to $s_\G$. For the most
general results see \cite{Totiktrans} and \cite{Totikarc}.

What is left, is to decide the asymptotic behavior
 at the endpoints of the arc components.
It turns out that this problem is closely related to
the asymptotic behavior away from the endpoints, but
for measures of the form $d\mu(x)=|z-z_0|^\a ds_\G(z)$, $\a>-1$,
and the aim of this paper is to find these asymptotic behaviors.
When $\mu$ is of the just specified form, then we shall show
(for the exact formulation see the next section),
\be
    \lim_{n \to \infty} n^{1+\alpha} \lambda_n(\mu, z_0) =  \frac{1}{(\pi \omega_{\Gamma}(z_0))^{\alpha + 1}} 2^{\alpha + 1}\Gamma \Big( \frac{\alpha + 1}{2} \Big) \Gamma \Big( \frac{\alpha + 3}{2} \Big)\label{s1}\ee
when $z_0$ is not the endpoint of an arc component of $\G$, while
at an endpoint
\[ \lim_{n \to \infty} n^{2\alpha + 2} \lambda_n(\mu,z_0) =  \frac{\Gamma(\alpha + 1)\Gamma(\alpha + 2)}{(\pi M(\G,z_0))^{2\alpha + 2}},\]
where $M(\G,z_0)$ is the limit of $\sqrt{|z-z_0|}\o_\G(z)$ as $z\to z_0$
along $\G$.

This paper uses some basic notions and results from potential theory.
See \cite{Gardiner}, \cite{Garnett}, \cite{Ransford} or
\cite{StahlTotik} for all the concepts we use and for the basic theory.
In particular, $\nu_\G$ will denote the equilibrium measure of
the compact set $\G$.

Since the asymptotics reflect the support of the measure,
in all such questions a global condition, stating that
the measure is not too small on any part of $\G$, is needed
(for example, if $\mu$ zero on any arc of $\G$, then
(\ref{s1}) does not hold any more). This global condition
is the regularity condition from \cite{StahlTotik}:
we say that $\mu$, with support $\G$,  belongs to the {\bf Reg}
class if
\[ \sup_{P_n}\left(\frac{\|P_n\|_\G}{\|P_n\|_{L^2(\mu)}}\right)^{1/n}\to 1\]
as $n\to\i$, where the supremum is taken for all polynomials
of degree at most $n$, and where $\|P_n\|_\G$
denotes the supremum norm on $\G$. The condition says that in the $n$-th root
sense the $L^\i(\mu)$ and $L^2(\mu)$-norms are almost the same.
The assumption $\mu\in {\bf Reg}$ is a very weak condition -- see \cite{StahlTotik}
for several reformulations as well as conditions on the measure
$\mu$ that implies $\mu\in {\bf Reg}$. For example,
if $\G$ consists of rectifiable Jordan curves and arcs
with arc measure $s_\G$,
then any measure $d\mu(z)=w(z)ds_\G(z)$ with $w(z)>0$ $s_\G$-almost
everywhere is regular in this sense.

Actually, it is not even needed that the support $\G$ of the measure
$\mu$ be a system of
Jordan curves or arcs, the main theorem
below holds for any $\G$ that is a finite union of
continua (connected compact sets). However, it is needed that
$z_0$ lies on a smooth arc $J$ of the outer boundary of $\G$:
the outer boundary of $\G$ is the boundary of the unbounded
connected component of $\ov\C\setm \G$. It is known that
the equilibrium measure $\nu_\G$ lives on the outer boundary, and
if $J$ is a smooth (say $C^1$-smooth)  arc on the outer boundary,
then on $J$ the equilibrium measure is absolutely continuous
with respect to the arc measure $s_J$ on $J$:
$d\nu_\G(z)=\o_\G(z)ds_J(z)$. We call this $\o_\G$ the equilibrium
density of $\G$.

The following theorem describes the asymptotics of the
Christoffel function at points that are different from
the endpoints of the arc-components/parts of $\G$,
see Figure \ref{fig-1} for illustration.

\begin{figure}
\centering
\includegraphics[scale=0.7]{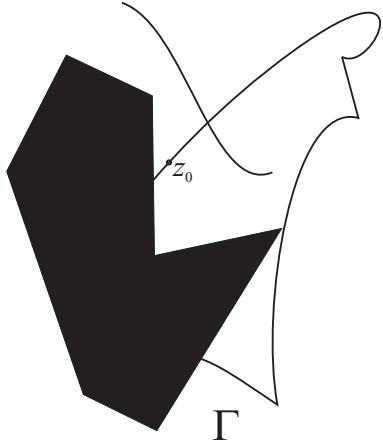}
\caption{\protect\label{fig-1} A typical position where Theorem
\ref{thmain} can be applied}
\end{figure}

\begin{thh}\label{thmain}
      Let the support $ \Gamma $ of a measure $\mu\in{\bf Reg}$ consist of finitely many continua,
    and let $ z_0$ lie on the outer boundary of $\G$.
     Assume that the intersection of $ \Gamma $ with a neighborhood
    of $z_0$ is $ C^2 $-smooth arc $J$ which contains $z_0$ in its
   (one-dimensional) interior. Assume also that in this neighborhood
     $ d\mu(z) = w(z)|z-z_0|^\alpha ds_{J}(z) $,
    where $ w $ is a strictly positive continuous function and $ \alpha > -1 $. Then
\begin{equation}\label{main_theorem_equation}
    \lim_{n \to \infty} n^{1+\alpha} \lambda_n(\mu, z_0) =  \frac{w(z_0)}{(\pi \omega_{\Gamma}(z_0))^{\alpha + 1}} 2^{\alpha + 1}\Gamma \Big( \frac{\alpha + 1}{2} \Big) \Gamma \Big( \frac{\alpha + 3}{2} \Big).
\end{equation}
\end{thh}

The second main theorem of this work is about the behavior of the Christoffel
function at an endpoint, see Figure \ref{fig-2}. If $z_0$ is an endpoint of a smooth arc $J$
on the outer boundary of $\G$, then at $z_0$ the equilibrium
density has a $1/\sqrt{|z-z_0|}$ behavior (see the proof
of Theorem \ref{thmainarc}), and we set
\be M(\G,z_0):=\lim_{z\to z_0,\ z\in \G} \sqrt{|z-z_0|}\o_\G(z).\label{mgdef}\ee

\begin{thh}\label{thmainarc} Let $\G$ and $\mu$ be as in Theorem \ref{thmain},
but now assume that the  intersection of $ \Gamma $ with a neighborhood
    of $z_0$ is $ C^2 $-smooth Jordan arc $J$ with one endpoint at $z_0$.
     Then
\be      \lim_{n \to \infty} n^{2\alpha + 2} \lambda_n(\mu,z_0) =  \frac{w(z_0)}{(\pi M(\G,z_0))^{2\alpha + 2}}\Gamma(\alpha + 1)\Gamma(\alpha + 2).\label{arceq}\ee
\end{thh}

These results can be used, in particular, if the measure
is supported on a finite union of intervals on the real line,
in which case the quantities $\o_\G(x)$ and
$M(\G,x)$ have a rather explicit form.
Let $\G=\cup_{j=0}^{k_0}[a_{2j},a_{2j+1}]$ with disjoint $[a_{2j},a_{2j+1}]$.
Then the equilibrium density of $\G$ is (see e.g. \cite[(40), (41)]{T2}
or \cite[Lemma 4.4.1]{StahlTotik})
\begin{equation}
\omega_{\G}(x)=\frac{\prod_{j=0}^{k_0-1}|x-
\lambda_{j}|}{\pi\sqrt{\prod_{j=0}^{2k_0+1}|x-a_{j}|}},\qquad x\in {\rm Int}(\G),
\label{omega}
\end{equation}
where $\lambda_{j}$ are the solutions of the system of equations
\begin{equation} \int_{a_{2k+1}}^{a_{2k+2}}\frac{\prod_{j=0}^{k_0-1}(t-
\lambda_{j})}{\sqrt{\prod_{j=0}^{2k_0+1}|t-a_{j}|}}dt=0
\label{omega1},\qquad k=0,\ldots k_0-1.\end{equation}
It can be easily shown that these
$\lambda_{j}$'s  are
uniquely determined and there is one $\lambda_{j}$ on every contiguous
interval
$(a_{2j+1},a_{2j+2})$.
Now if $a$ is one of the endpoints of the intervals
of $\G$, say $a=a_{j_0}$, then
\begin{equation}
M(\G,a)=\frac{\prod_{j=0}^{k_0-1}|a-
\lambda_{j}|}{\pi\sqrt{\prod_{j=1,\ j\not=j_0}^{2k_0}|a-a_{j}|}}.
\label{finite}
\end{equation}

\begin{figure}
\centering
\includegraphics[scale=0.7]{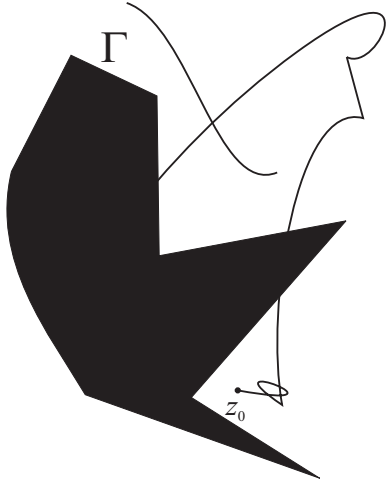}
\caption{\protect\label{fig-2} A typical position where Theorem
\ref{thmainarc} can be applied}
\end{figure}

This whole work is dedicated to proving Theorem \ref{thmain}
and Theorem \ref{thmainarc}. Actually, the latter will be
a relatively easy consequence of the former one, so the main
emphasis will be to prove Theorem \ref{thmain}.
The main line of reasoning will be the following. We start from
some known facts for simple measures like $|x|^\a dx$ on the real line,
and get some elementary results for a model case on the
unit circle via a transformation. Then we prove from these simple
cases that
Theorem
\ref{thmain} is true for lemniscate sets, i.e. level sets of polynomials.
This part will use the polynomial mapping in question to
transform the already known result to the given lemniscate.
Then we prove the theorem for finite unions of Jordan curves.
Recall that a Jordan curve is a homeomorhic image
of a circle, while a Jordan arc is a homeomorhic image of a segment.
From the point of view of finding the asymptotics of
Christoffel functions there is a
big difference between arcs and curves: Jordan curves have interior
and can be exhausted by lemniscates, so the polynomial
inverse image method of \cite{T2} is applicable for them,
while for Jordan arc that method cannot be applied.
Still, the pure Jordan curve case is used
when we go over to a $\G$ which may have arc components,
namely it is used in the lower estimate. The upper estimate
is the most difficult part of the proof; there Bessel functions
enter the picture, and a discretization technique is developed
where  the discretization
of the equilibrium measure of $\G$  is done using the zeros of appropriate
Bessel functions combined with another discretization
based on uniform distribution. Once the case of Jordan
curves and arcs have been settled, the proof
of Theorem \ref{thmain} will easily follow by
approximating a general $\G$ by a family of Jordan curves
and arcs.

\sect{Tools}\label{sectools}
In what follows, $\|\cdot\|_K$ denotes the supremum norm on a set $K$, and $s_\G$
the arc measure on $\G$ (when $\G$ consists of smooth Jordan arcs or curves).

We shall rely on some basic notions and facts from logarithmic potential theory.
See the books  \cite{Gardiner}, \cite{Garnett}, \cite{Ransford} or
\cite{SaffTotik}
for detailed discussion.

We shall often use the trivial fact that if $\mu, \nu$ are
two Borel measures, then $\mu\le \nu$ implies
 \( \lambda_n(\mu, x) \leq \lambda_n(\nu, x) \) for all $x$. It is
 also trivial that $\l_n(\mu,z)\le \mu(\C)$ (just use the identically 1 polynomial as a
 test function
 in the definition of $\l_n(\mu,z)$).

Another frequently used fact is the following: if $\{n_k\}$ is a subsequence of
the natural numbers such that $n_{k+1}/n_k\to1$ as $k\to\i$, then for any $\k>0$
\be \liminf_{n\to\i} n^\k\l_{n}(\mu,x)=\liminf_{k\to\i} n_k^\k\l_{n_k}(\mu,x),\label{liminf}\ee
and
\be \limsup_{n\to\i} n^\k\l_n(\mu,x)=\limsup_{k\to\i} n_k^\k\l_{n_k}(\mu,x).\label{limsup}\ee
In fact, since  $\l_n(\mu,x)$ is a monotone decreasing function of $n$, for $n_k\le n\le n_{k+1}$
we have
\[\left(\frac{n}{n_{k+1}}\right)^\k n_{k+1}^\k\l_{n_{k+1}}(\mu,x)\le n^\k\l_n(\mu,x)\le \left(\frac{n}{n_{k}}\right)^\k
 n_{k}^\k\l_{n_{k}}(\mu,x),\]
and both claims follow because
 $n/n_k$ and $n/n_{k+1}$ tend to 1 as $n$ (or $n_k$) tends to infinity.

\subsection{Fast decreasing polynomials}
The following lemmas on the existence of fast decreasing
polynomials will be a constant tool in the proofs.

\begin{prop}\label{fdp}
    Let $ K $ be a compact subset on $ \mathbb{C} $, $ \Omega $ the unbounded complement of $ \mathbb{C} \setminus K $ and let $ z_0 \in \partial \Omega $. Suppose that there is a disk in $ \Omega $ that contains $ z_0 $ on its boundary. Then, for every $ \gamma > 1 $, there are constants $ c_\gamma, C_\gamma $, and for every $ n \in \mathbb{N} $ polynomials $ S_{n,z_0,K} $ of degree at most $ n $ such that $ P_n(z_0) = 1 $, $ |S_{n,z_0,K}(z)| \leq 1 $ for all $ z \in K $ and
\begin{equation}\label{kineq}
    |S_{n,z_0,K}(z)| \leq C_\gamma e^{-n c_\gamma |z-z_0|^{\gamma}}, \quad z \in K.
\end{equation}
\end{prop}
For details, see \cite[Theorem 4.1]{Totiktrans}. This theorem will often
be used in the following form.

\begin{cor}\label{fdp_beta}
     With the assumptions of Proposition \ref{fdp}
     for every $ 0 < \tau < 1 $, there exists constants $ c_\tau, C_\tau, \tau_0>0$
     and for every $ n \in \mathbb{N} $ a polynomial $ S_{n, z_0, K} $ of degree $ o(n) $
      such that $ S_{n,z_0,K}(z_0) = 1 $, $ |S_{n,z_0,K}(z)| \leq 1 $ for all $ z \in K $, and
\begin{equation}\label{asd}
    |S_{n, z_0, K}(z)| \leq C_\tau e^{- c_\tau n^{\tau_0}}, \quad |z-z_0| \geq n^{-\tau}.
\end{equation}
\end{cor}

\begin{proof}.
  Let $ 0 < \varepsilon $ be sufficiently small and select
  $\g>1$ so that
     $ 1 - \varepsilon - \tau \gamma > 0 $.
     Lemma \ref{fdp} tells us that there is a polynomial $ P_n $ with $ \deg(P_n) \leq n^{1-\varepsilon} $ such that
\[
    |P_n(z)| \leq C_\g e^{-c_\g n^{1-(\varepsilon+\tau \gamma)}}, \quad |z-z_0| \geq n^{-\tau},
\]
and this proves the claim with $S_{n,z_0,K}=P_n$.
\end{proof}

There is a version of Lemma \ref{fdp} where the
decrease is not exponentially small, but starts much earlier
than in Lemma \ref{fdp}.

\begin{prop}\label{fdpalt} Let $K$ be as in Proposition \ref{fdp}.
Then, for every $\b<1$, there are constants $c_\b,C_\b>0$, and for every
$n=1,2,\ldots$ polynomials $P_n$ of degree at most
$n$ such that $P_n(z_0)=1$, $|P_n(z)|\le 1$ for $z\in K$ and
\be |P_n(z)|\le C_\b e^{-c_\b(n|z-z_0|)^\b},\qquad z\in K.\label{sest}\ee
\end{prop}
See \cite[Lemma 4]{Totikszego}.

It will be convenient to use these results when $n>1$ is not necessarily
integer (formally one has to take the integral part
of $n$, but the estimates will hold with
possibly smaller constants in the exponents).

\subsection{Polynomial inequalities}
We shall also  need some inequalities
for polynomials that are used several times
in the rest of the paper.

We start with a Bernstein-type inequality.

\begin{lem}\label{bernstein}
    Let  $J$ be a $ C^2 $ closed Jordan arc and  $ J_1 $ a closed subarc of
    $ J $ not having common endpoint with $ J $.
    Then, for every $ D > 0 $, there is a constant $ C_D $, such that
\[
    |P_{n}^{'}(z)| \leq C_D n \| P_n \|_J, \qquad \operatorname{dist}(z,J_1) \leq D/n,
\]
holds for any polynomials $P_n$ of degree $n=1,2,\ldots$.
\end{lem}
See \cite[Corollary 7.4]{Totiktrans}.

Next, we continue with a Markov-type inequality.

\begin{lem} \label{corMarkov} Let $K$ be a continuum.
If $Q_n$ is a polynomial of degree
at most $n=1,2,\ldots$, then
\be \|Q_n'\|_{K}\le \frac{e}{2\c(K)}n^2\|Q_n\|_K,\label{09}\ee
where $\c(K)$ denotes the logarithmic capacity of $K$.

In particular, if $K$ has diameter 1, then
\be \|Q_n'\|_{K}\le 2e n^2\|Q_n\|_K.\ee
\end{lem}
For (\ref{09}) see \cite[Theorem 1]{Pommerenke}, and for the last statement
note that if $K$ has diameter 1, then its capacity is at least $1/4$
(\cite[Theorem 5.3.2(a)]{Ransford}).

Next, we prove a Remez-type inequality.

\begin{lem}\label{Remez}
Let $\G$ be a $C^1$ Jordan curve or arc, and assume
that for every $n=1,2,\ldots$, $J_n$ is a subarc of $\G$,
and $J_n^*$ is a subset of $J_n$ such that
\[s_\G(J_n\setm J_n^*)=o(n^{-2})s_\G(J_n),\]
where $s_\G$ denotes the
arc-length measure on $\G$.
Then, for any sequence $\{Q_n\}$of polynomials of degree at most
$n=1,2,\ldots$, we have
\be \|Q_n\|_{J_n}=(1+o(1))\|Q_n\|_{J_n^*}.\label{ineq1}\ee
\end{lem}
\proof. It is clear from the $C^1$ property that
$s_\G(J_n)\sim {\rm diam} (J_n)$ uniformly in $J_n$
(meaning that the ratio of the two sides lies in between
two positive constants).

Make a linear transformation $z\to Cz$ such that, after this
transformation, the arc $\tilde J_n$ that we obtain from $J_n$
has diameter 1. Under this transformation $J_n^*$ goes into
a subset $\tilde J_n^*$ of $\tilde J_n$ for which
\be s_{\tilde J_n}(\tilde J_n\setm \tilde J_n^*)=o(n^{-2})s_{\tilde J_n}(\tilde J_n),
\label{assu}\ee
and $Q_n$ changes into a polynomial $\tilde Q_n$ of degree at
most $n$. (\ref{ineq1}) is clearly equivalent to its $\tilde{}$-version.

Let $M=\|\tilde Q_n\|_{\tilde J_n}$. By Lemma \ref{corMarkov},
the absolute value of $\tilde Q_n'$ is bounded on $\tilde J_n$ by $2e n^2M$,
hence if $z,w\in \tilde J_n$, then
\be |\tilde Q_n(z)-\tilde Q_n(w)|\le 2en^2M s_{\tilde J_n}(\ov{zw}),\label{bbb}\ee
where $\ov{zw}$ is the arc of $\tilde J_n$ lying in between $z$ and $w$.
By the assumption (\ref{assu}) for every $z\in \tilde J_n$
there is a $w\in \tilde J_n^*$ with
\[ s_{\tilde J_n}(\ov{zw})=o(n^{-2})s_{\tilde J_n}(\tilde J_n)
=o(n^{-2})\]
because $s_{\tilde J_n}(\tilde J_n)\sim {\rm diam} (\tilde J_n)=1$.
Choose here $z\in \tilde J_n$ such that $|\tilde Q_n(z)|=M$.
Since $|\tilde Q_n(w)|\le \|\tilde Q_n\|_{\tilde J_n^*}$,
we get from (\ref{bbb})
\[M=|\tilde Q_n(z)|\le \|\tilde Q_n\|_{\tilde J_n^*}
+ o(1)M,\]
and the claim follows.\endproof

We shall frequently use the following,
so called Nikolskii-type inequalities for power type weights.
In it we write that a Jordan arc is $C^{1+}$-smooth  if
there is a $\t>0$ such that the arc in question is $C^{1+\t}$-smooth.

\begin{lem}\label{Niklemma} Let $J$ be a $C^{1+}$-smooth Jordan arc and let $J^*\subset J$ be a subarc of $J$ which has no common endpoint
with $J$. Let $z_0\in J$ be a fixed point, and for $\a>-1$ define the measure
$\nu_\a$ on $J$ by $d\nu_\a(u)=|u-z_0|^\a ds_J(u)$. Then there is
a constant $C$ depending only on $\a,J$ and $J^*$ such that for any polynomials
$P_n$ of degree at most $n=1,2,\ldots$ we have
\be \|P_n\|_{J^*}\le Cn^{(1+\a)/2}\|P_n\|_{L^2(\nu_\a)},\label{nik1}\ee
if $\a \ge 0$, and
\be \|P_n\|_{J^*}\le Cn^{1/2}\|P_n\|_{L^2(\nu_\a)},\label{nik2}\ee
if $-1<\a<0$.

The same is true if $d\nu_\a(u)=w(u)|u-z_0|^\a ds_J(u)$
with some strictly positive and continuous $w$.
\end{lem}
\proof. In view of \cite[Lemmas 3.8 and Corollary 3.9]{Varga}
(use also that $\nu_\a$ is a doubling weight
in the sense of \cite{Varga})
uniformly in $z\in J^*$ we have  for large $n$ the relation
\[\l_n(\nu_\a,z)\sim \nu_\a(l_{1/n}(z)),\]
where $A\sim B$ means that the ratio lies in between
two constants, and where
$l_{1/n}(z)$ is the arc of $J$ consisting of those
points of $z$ that lie of distance $\le 1/n$ from $z$.
If $\a\ge 0$, then
 \[\nu_\a(l_{1/n}(z))\ge \frac{c}{n^{1+\a}},\]
while for $-1<\a<0$
 \[\nu_\a(l_{1/n}(z))\ge \frac{c}{n},\]
with some positive constant $c$ which depends only
on $\a,J$ and $J^*$. Therefore, we have for all $z\in J^*$
the inequality
\be \l_n(\nu_\a,z)\ge \frac{c}{n^{1+\a}}\label{nik4}\ee
if $\a \ge 0$ and
\be\l_n(\nu_\a,z)\ge \frac{c}{n}\label{nik3}\ee
when $-1<\a<0$.

For example, (\ref{nik4}) means that
if $\a\ge 0$ and $|P_n(z)|=1$ for some $z\in J^*$, then
necessarily
\[\frac{n^{1+\a}}{c}\int_J |P_n|^2d\nu_\a\ge 1,\]
which is equivalent to saying that  for any $P_n$ and $z\in J^*$
\[\frac{n^{1+\a}}{c}\int_J |P_n|^2d\nu_\a\ge |P_n(z)|^2,\]
and this is  (\ref{nik1}). In a similar manner,
(\ref{nik2}) follows from (\ref{nik3}).

It is clear that this proof does not change if
$\nu_\a$ is as in the last sentence of the lemma.
\endproof

\begin{lem}\label{Nikolskii} If $\a>-1$, then there is a constant $C_\a$ such that
for any polynomial $P_n$ of degree at most $n$ the inequality
\be \|P_n\|_{[-1,1]}\le C_\a n^{(1+\a^{*})/2}\left(\int_{-1}^1|P_n(x)|^2|x|^\a dx\right)^{1/2}
\label{N2}\ee
holds with $\a^{*}=\max(1,\a)$.\end{lem}

\proof. We follow the
preceding proof, but now both $J$ and $J^*$ agree
with  $[-1,1]$.

Let  $J=J^*=[-1,1]$, $z_0=0$,
 $\D_n(z)=1/n^2$ if $ z\in [-1, -1+1/n^2]$ or $z\in [1-1/n^2,1]$,
and set $\D_n(z)=\sqrt{1-z^2}/n$ if $z\in [-1+1/n^2,1-1/n^2]$.
If now $l_{1/n}(z)$ is the interval $[z-\D_n(z),z+\D_n(z)]$ intersected
with $[-1,1]$, then
\cite[Lemmas 3.8 and Corollary 3.9]{Varga}
state  that for  $d\nu_\a(x)=|x-z_0|^\a dx=|x|^\a dx$ on $[-1,1]$
we have
\[ \l_n(\nu_\a,z)\sim \nu_\a(l_{1/n}(z)).\]
If $\a\ge 0$, then
 \[\nu_\a(l_{1/n}(z))\ge c\min\left(\frac{1}{n^2},\frac{1}{n^{1+\a}}\right),\]
while for $-1<\a<0$
 \[\nu_\a(l_{1/n}(z))\ge \frac{c}{n^2},\]
with some positive constant $c$.
Hence,
\[\l_n(\nu_\a,z)\ge \frac{c}{n^{2}}\]
if $-1<\a\le 1$, while
\[\l_n(\nu_\a,z)\ge \frac{c}{n^{1+\a}}\]
if $\a\ge 1$,
from which (\ref{N2}) follows exactly as before.\endproof

The Nikolskii inequalties can be combined with the following
estimate to get an upper bound for the extremal polynomials
that produce $\l_n(\mu,z)$.

\begin{lem}\label{a_priori_estimate} With the assumptions of Theorem \ref{thmain}
we have
\[
    \lambda_n(\mu, z_0) \leq C n^{-(\alpha+1)}
\]
with some constant $C$ that is independent of $n$.
\end{lem}
\proof. Just use the polynomials $S_{n,z_0,\G}$ from
Proposition \ref{fdpalt} with $\b=1/2$ and $K=\G$.
Let $\d>0$ be so small that in the $\d$-neighborhood of $z_0$
we have the $ d\mu(z) = w(z)|z-z_0|^\alpha ds_{\Gamma}(z) $
representation for $\mu$. Outside this $\d$-neighborhood
$|S_{n,z_0,\G}|$ is smaller than $C_{\b}\exp(-c_{\b}(n\d)^{1/2})$, so
\[\int |S_{n,z_0,\G}|^2d\mu\le C\int e^{-2c_{\b}(n|t|)^{1/2}}|t|^\a dt +C
e^{-2c_\b (n\d)^{1/2}}\le Cn^{-\a-1},\]
which proves the claim.
\endproof

We close this section with the classical Bernstein-Walsh lemma,
see  \cite[p. 77]{Walsh}.

\begin{lem}\label{BW} Let $K\subset \C$ be a compact subset of
positive logarithmic capacity, let $\O$ be the unbounded component
of $\ov\C\setm K$, and $g_\O$ the Green's function of this
unbounded component with pole at infinity. Then, for polynomials
$P_n$ of degree at most $n=1,2,\ldots$, we have for any $z\in \C$
\[|P_n(z)|\le e^{ng_\O(z)}\|P_n\|_K.\]
\end{lem}

\sect{The model cases}

\subsection{Measures on the real line}\label{sectmodel1}
Our first goal is to establish asymptotics for the Christoffel function at \( 0 \) with respect to the measure \( d\mu(x) = |x|^\alpha dx \), \( x \in [-1,1] \). We do this by transforming some
 previously known results.

 In what follows, for simpler notations, if $d\mu(x)=w(x)dx$, then
 we shall write $\l_n(w(x),z)$ for $\l_n(\mu,z)$.

\begin{prop}\label{model_1}
For $\a>-1$ we have
\begin{equation}
    \lim_{n \to \infty} n^{2\alpha + 2} \lambda_n\left({|x|^\a\restr{[0,1]}},0\right) = \Gamma(\alpha + 1)\Gamma(\alpha + 2).
\end{equation}
\end{prop}

\proof. It follows from \cite[(1.10)]{Lubinsky2} or \cite[Theorem 4.1]{Lubinsky1}
that
    \begin{equation}\label{rer}
    \lim_{n \to \infty} n^{2\alpha + 2} \lambda_n\left({(1-x)^\a\restr{[-1,1]}}, 1\right) = 2^{\alpha+1}\Gamma(\alpha + 1) \Gamma(\alpha + 2),
    \end{equation}
from which the claim is an immediate consequence if we apply the linear transformation $x\to (1-x)/2$.\endproof

\begin{prop}\label{model_2}
    For  \( \alpha > -1 \) we have
    \begin{equation}\label{nincs}
    \lim_{n \to \infty} n^{\alpha + 1} \lambda_n\left({|x|^\a\restr{[-1,1]}},0\right) = L_\a,
        \end{equation}
where
\be L_\a:=2^{\alpha + 1} \Gamma\Big( \frac{\alpha + 1}{2} \Big) \Gamma\Big( \frac{\alpha + 3}{2} \Big).\label{lalpha}\ee

\end{prop}

\proof. Let us agree that in this proof, whenever we write $P_n,R_n$ etc. for polynomials, then it is understood that
the degree is at most $n$.

We use that (for continuous $f$)
    \begin{equation}\label{model_2_transform}
    \int_{0}^{1} f(x) |x|^\alpha dx  = \int_{-1}^{1} f(x^2) |x|^{2\alpha + 1} dx.
    \end{equation}

Assume first that \( P_{2n} \) is extremal for \( \lambda_{2n}\left({|x|^\alpha \restr{[-1,1]}},0\right) \), i.e.
$P_{2n}(0)=1$ and
\[\int_{-1}^1|P_{2n}(x)|^2|x|^\a dx=\lambda_{2n}\left({|x|^\alpha \restr{[-1,1]}},0\right).\]
Define
    \[
    R_n(x) = \frac{P_{2n}(x) + P_{2n}(-x)}{2}.
    \]
Then \( R_n(0) = 1 \), and \( R_n \) is a polynomial in \( x^2 \), hence
     \(  R_n(x)=R_{n}^{*}(x^2)   \) with some polynomial $R_n^*$,
     for which \( R_{n}^{*}(0) = 1 \) and \( \deg(R_{n}^{*}) \le  n \). Now we have
\bean     \int_{-1}^{1} |R_{n}(x)|^2 |x|^\alpha dx  = \int_{-1}^{1} |R_{n}^{*}(x^2)|^2|x|^\alpha dx
     &=& \int_{0}^{1} |R_{n}^{*}(x)|^2 |x|^{\frac{\alpha - 1}{2}} dx\\
     &\geq& \lambda_n\left( {|x|^{\frac{\alpha - 1}{2}}\restr{[0,1]}}, 0\right).\eean
     With the Cauchy-Schwarz inequality and with the symmetry of the measure \( |x|^\alpha dx \), we have
    \begin{align*}
    \int_{-1}^{1} |R_n(x)|^2|x|^\alpha dx & \leq \frac{1}{4} \int_{-1}^{1}\Big( |P_{2n}(x)|^2 + 2|P_{2n}(x)||P_{2n}(-x)| + |P_{2n}(-x)|^2 \Big) |x|^\alpha dx \\
    & \leq \frac{1}{2} \int_{-1}^{1} |P_{2n}(x)|^2 |x|^\alpha dx \\ &
     + \frac{1}{2} \Bigg( \int_{-1}^{1} |P_{2n}(x)|^2|x|^\alpha dx \Bigg)^{1/2} \Bigg( \int_{-1}^{1} |P_{2n}(-x)|^2|x|^\alpha dx \Bigg)^{1/2} \\
    & = \int_{-1}^{1} |P_{2n}(x)|^2|x|^\alpha dx
     = \lambda_{2n}\left({|x|^\alpha \restr{[-1,1]}}, 0\right).
    \end{align*}
    Combining these two estimates, we obtain
  \[\lambda_n\left({|x|^{\frac{\alpha - 1}{2}} \restr{[0,1]}}, 0\right) \leq \lambda_{2n}\left({|x|^\alpha \restr{[-1,1]}}, 0\right).\]

On the other hand, if now \( P_n \) is extremal for \( \lambda_n\left({|x|^{\frac{\alpha - 1}{2}} \restr{[0,1]}}, 0\right) \), then
\bean     \lambda_n\left({|x|^{\frac{\alpha -1}{2}}\restr{[0,1]}}, 0\right)  = \int_{0}^{1} |P_n(x)|^2 |x|^{\frac{\alpha - 1}{2}} dx
  &=& \int_{-1}^{1} |P_n(x^2)|^2 |x|^\alpha dx \\
  &\geq& \lambda_{2n}\left({|x|^\alpha\restr{[-1,1]}}, 0\right),\eean
therefore we actually have the equality
\be     \lambda_n\left({|x|^{\frac{\alpha -1}{2}}\restr{[0,1]}}, 0\right)  =  \lambda_{2n}\left({|x|^\alpha\restr{[-1,1]}}, 0\right),\label{mod01}\ee
from which the claim follows via Proposition \( \ref{model_1} \) (see also (\ref{liminf}) and (\ref{limsup}) with $n_k=2k$).

Note also that this proves also that if $P_n(x)$ is the $n$-degree extremal polynomials for
the measure ${|x|^{\frac{\alpha -1}{2}}\restr{[0,1]}}$, then $P_n(x^2)$ is
the $2n$-degree extremal polynomial for the measure
${|x|^\alpha\restr{[-1,1]}}$.

\endproof

\subsection{Measures on the  unit circle}

Let $ \mu_{\mathbb{T}} $ be the measure on the unit circle $ \mathbb{T} $ defined by $ d\mu_{\mathbb{T}}(e^{it}) = w_{\mathbb{T}} (e^{it}) dt $, where
\begin{equation}\label{circle_pre_model_case}
  w_{\mathbb{T}}(e^{it}) = \frac{|e^{2it}+1|^{\alpha}}{2^{\alpha}} \frac{|e^{2it}-1|}{2} , \qquad t \in [-\pi, \pi).
\end{equation}
We shall prove
\be
    \lim_{n \to \infty} n^{\alpha + 1} \lambda_n(\mu_\mathbb{T}, e^{i\pi/2}) = 2^{\a+1}L_\a\label{mt}
\ee
where $L_\a$ is from (\ref{lalpha}),
by transforming the measure $ \mu_\mathbb{T} $ into a measure $ \mu_{[-1,1]} $ supported on the interval $ [-1,1] $ and comparing the Christoffel functions for them.
With the transformation $ e^{it} \to \cos t $, we have
\[
    \int_{-\pi}^{\pi} f(\cos t) w_\mathbb{T}(e^{it})dt = 2 \int_{-1}^{1} f(x) w_{[-1,1]}(x) dx,
\]
where \[ w_{[-1,1]}(x) = |x|^\alpha.\]
Set $ d\mu_{[-1,1]}(x) = w_{[-1,1]}(x) dx $.

 Let $ P_n $ be the extremal polynomial for $ \lambda_n(\mu_{[-1,1]},0) $ and define
\[
    S_n(e^{it}) = P_n(\cos t) \Bigg( \frac{1+e^{i(t-\pi/2)}}{2} \Bigg)^{\lfloor \eta n \rfloor} e^{in(t-\pi/2)},
\]
where $ \eta > 0 $ is arbitrary. This $ S_n $ is a polynomial of degree $ 2n + \lfloor \eta n \rfloor $ with $ S_n(e^{i\pi/2}) = 1 $. For any fixed $0<\d<1$
\begin{equation}\label{circle_upper_1}
\begin{aligned}
    \int_{\pi/2 - \delta}^{\pi/2 + \delta} |S_n(e^{it})|^2 w_\mathbb{T}(e^{it}) dt & \leq \int_{\pi/2 - \delta}^{\pi/2 + \delta} |P_n(\cos t)|^2 w_\mathbb{T}(e^{it}) dt \\
    & \leq \int_{-1}^{1} |P_n(x)|^2 w_{[-1,1]}(x) dx \\
    & = \lambda_n(\mu_{[-1,1]},0).
\end{aligned}
\end{equation}
To estimate the corresponding integral over the intervals $ [-\pi, \pi/2 - \delta] $ and $ [\pi/2 + \delta, \pi] $, notice that
\be
    \max_{t \in [-\pi, \pi] \setminus [\pi/2 - \delta, \pi/2 + \delta]} \Bigg| \frac{1+e^{i(t-\pi/2)}}{2} \Bigg|^{\lfloor \eta n \rfloor} = O(q^n)
\label{qq}\ee
for some $ q < 1 $.
From  Lemma \ref{Nikolskii}
we obtain
\[
    \| P_n \|_{[-1,1]} \leq Cn^{{1+|\a|}/2} \| P_n \|_{L^2(\mu_{[-1,1]})} \leq Cn^{1+|\a|/2},\]
and so
\[    \left(\int_{-\pi}^{\pi/2 - \delta}+\int_{\pi/2 + \delta}^\pi\right) |S_n(e^{it})|^2 w_\mathbb{T}(e^{it}) dt
=O(n^{1+|\a|/2}q^n)=o(n^{-\a-1}).\]
Therefore, using this $S_n$ as a test polynomial for $\lambda_{\deg(S_n)}(\mu_\mathbb{T}, e^{i\pi/2})$ we  conclude
\[
    \lambda_{\deg(S_n)}(\mu_\mathbb{T}, e^{i\pi/2}) \leq  \lambda_n(\mu_{[-1,1]},0)
    +o(n^{-\a-1}),
\]
and so
\begin{align*}
    \limsup_{n \to \infty} \big( 2n + \lfloor \eta n \rfloor \big)^{\alpha + 1} \lambda_{2n + \lfloor \eta n \rfloor}(\mu_\mathbb{T}, e^{i\pi/2}) & \leq \limsup_{n \to \infty} \big( 2 + \lfloor \eta n \rfloor /n \big)^{\alpha + 1} n^{\alpha + 1} \lambda_n(\mu_{[-1,1]}, 0) \\
    & = (2+\eta)^{\alpha + 1} L_\a,
\end{align*}
where we used Proposition  \ref{model_2} for the measure $\mu_{[-1,1]}$.

Since $ \eta > 0 $ was arbitrary,
\be
    \limsup_{n \to \infty} n^{\alpha + 1} \lambda_n(\mu_{\mathbb{T}}, e^{i\pi/2}) \leq 2^{\alpha + 1} L_\a
\label{lll}\ee
follows (see also (\ref{limsup})).

Now to prove the matching lower estimate,
let $S_{2n}(e^{it}) $ be the extremal polynomial for $ \lambda_{2n}(\mu_\mathbb{T}, e^{i\pi/2}) $. Define
\[
    P_{n}^{*}(e^{it}) = S_{2n}(e^{it}) \Bigg( \frac{1+e^{i(t-\pi/2)}}{2} \Bigg)^{2\lfloor \eta n \rfloor} e^{-(n+\lfloor \eta n \rfloor)i(t-\pi/2)}
\]
and $ P_n(\cos t) = P_{n}^{*}(e^{it}) + P_{n}^{*}(e^{-it}) $. Note that $ P_n(\cos t) $ is a polynomial in $ \cos t $ of $ \deg(P_n) \leq n + \lfloor \eta n \rfloor $ and $ P_n(0) = 1 $. With it we have
\begin{equation}\label{circle_lower_1}
    \lambda_{\deg(P_n)}(\mu_{[-1,1]},0) \leq \int_{-1}^{1} |P_n(x)|^2 w_{[-1,1]}(x) dx = \frac{1}{2} \int_{-\pi}^{\pi} |P_n(\cos t)|^2 w_\mathbb{T}(e^{it}) dt.
\end{equation}
First, we claim that for every fixed $ 0 < \delta < 1$
\begin{equation}\label{circle_lower_pn}
\begin{aligned}
    |P_n(\cos t)|^2 & = |P_{n}^{*}(e^{it})|^2 + O(q^n), \quad t \in [\pi/2-\delta,\pi/2 + \delta], \\
    |P_n(\cos t)|^2 & = |P_{n}^{*}(e^{-it})|^2 + O(q^n), \quad t \in [-\pi/2 - \delta, -\pi/2 + \delta], \\
    |P_n(\cos t)|^2 & = O(q^n) \quad \textnormal{otherwise,}
\end{aligned}
\end{equation}
hold for some $ q < 1 $. Indeed,
\[
    |P_n(\cos t)|^2 = |P_{n}^{*}(e^{it}) + P_{n}^{*}(e^{-it})|^2  \leq |P_{n}^{*}(e^{it})|^2 + 2|P_{n}^{*}(e^{it})||P_{n}^{*}(e^{-it})| + |P_{n}^{*}(e^{-it})|^2.
\]
If we apply Lemma \ref{Niklemma} to two subarcs (say of length
$5\pi/4$) of $\mathbb{T}$
that contain the upper, resp. the lower half of the unit circle, then we obtain
that
\[
    \| P_n^* \|_{\mathbb{T}}\le \|S_{2n}\|_{\mathbb{T}}
 \leq Cn^{(1+|\a|)/2}\| S_{2n} \|_{L^2(\mu_{\mathbb{T}})}\le Cn^{(1+|\a|)/2}.
\]
 Therefore (use (\ref{qq}))
\[
    |P_{n}^{*}(e^{it})| \leq C  q^n n^{(1+|\a|)/2}, \quad t \in [-\pi,\pi] \setminus [\pi/2 - \delta, \pi/2 + \delta].
\]
These imply (\ref{circle_lower_pn}). \\

Now we have
\begin{align*}
    \int_{-\pi}^{\pi} |P_n(\cos t)|^2 w_\mathbb{T}(e^{it}) dt & = \Bigg( \int_{\pi/2 - \delta}^{\pi/2 + \delta}  + \int_{-\pi/2 - \delta}^{-\pi/2 + \delta} \Bigg) |P_n(\cos t)|^2 w_\mathbb{T}(e^{it}) dt \\
    & + \Bigg(\int_{-\pi}^{-\pi/2 - \delta} + \int_{-\pi/2 + \delta}^{\pi/2 - \delta} + \int_{\pi/2 + \delta}^{\pi} \Bigg) |P_n(\cos t)|^2 w_\mathbb{T}(e^{it}) dt.
\end{align*}
(\ref{circle_lower_pn}) tells us that the last three terms  are $ O(q^n) $. For the other two terms we have, again by (\ref{circle_lower_pn}),
\begin{align*}
    \int_{\pi/2 - \delta}^{\pi/2 + \delta} |P_n(\cos t)|^2 w_\mathbb{T}(e^{it}) dt &=
    \int_{\pi/2 - \delta}^{\pi/2 + \delta} |P_n^*(e^{it})|^2 w_\mathbb{T}(e^{it}) dt + O(q^n) \\
    &\leq \int_{\pi/2 - \delta}^{\pi/2 + \delta} |S_{2n}(e^{it})|^2 w_\mathbb{T}(e^{it}) dt + O(q^n) \\
    & \leq \lambda_{2n}(\mu_\mathbb{T},e^{i\pi/2})+ O(q^n)
\end{align*}
and similarly,
\[ \int_{-\pi/2 - \delta}^{-\pi/2 + \delta}|P_n(\cos t)|^2 w_\mathbb{T}(e^{it}) dt \leq \lambda_{2n}(\mu_\mathbb{T},e^{i\pi/2})+ O(q^n) .\]
 Combining these estimates with (\ref{circle_lower_1}), we can conclude
\[
    \lambda_{\deg(P_n)}(\mu_{[-1,1]},0) \leq \lambda_{2n}(\mu_\mathbb{T},e^{i\pi/2})
    + O(q^n) ,
\]
therefore
\begin{align*}
    \liminf_{n \to \infty} \deg(P_n)^{\alpha + 1} & \lambda_{\deg(P_n)}(\mu_{[-1,1]},0) \leq \liminf_{n \to \infty} (n + \lfloor \eta n \rfloor)^{\alpha + 1} \big( \lambda_{2n}(\mu_\mathbb{T},e^{i\pi/2})+ O(q^n)  \big)  \\
    & \leq \liminf_{n \to \infty} (1 + \lfloor \eta n \rfloor /n)^{\alpha + 1}\frac{1}{2^{\alpha+1}} (2n)^{\alpha + 1} \lambda_{2n}(\mu_\mathbb{T},e^{i\pi/2}).
\end{align*}
From this, in view of  Proposition \ref{model_2} and (\ref{liminf}), it follows that
\begin{align*}
 (1+\eta)^{-(\a+1)}   2^{\alpha + 1}L_\a \le \liminf_{n \to \infty} \lambda_n(\mu_{\mathbb{T}}, e^{i\pi/2}),
\end{align*}
and upon letting $ \eta\to 0 $ we obtain
\be
    2^{\alpha + 1}L_\a \leq \liminf_{n \to \infty} \lambda_n(\mu_{\mathbb{T}}, e^{i\pi/2}).
\label{llll}\ee
This and (\ref{lll}) verify (\ref{mt}).
\bigskip

Finally, let
\[ d\mu_\a(e^{it}) = |e^{it} - i|^{\alpha} dt.\]
Let us write
$|e^{it} - i|^{\alpha}$ in the form
\[|e^{it} - i|^{\alpha}=w(e^{it}) w_\mathbb{T}(e^{it}).  \]
Then $w$ is continuous in a neighborhood of $e^{i\pi/2}$
and it has value 1 at $e^{i\pi/2}$.
Let $ \tau > 0 $ be arbitrary, and choose $ 0<\delta <1 $ in such a way that
\[
    \frac{1}{1+\tau}\leq w(e^{it}) \leq (1+\tau) , \qquad t \in [\pi/2 - \delta, \pi/2 + \delta].
\]
If we now carry out the preceding arguments with this $\d$ and with this $\mu_\a$ replacing
everywhere $\mu_{\mathbb{T}}$, then we get that in (\ref{lll}) the limsup is
at most $(1+\tau)2^{\alpha + 1}L_\a$,
while in (\ref{llll}) the liminf is at least  $(1+\tau)^{-1}2^{\alpha + 1}L_\a$.
Since $\tau>0$ can be arbitrarily chosen, this shows that
\begin{equation}\label{unit_circle_result}
    \lim_{n \to \infty} n^{\alpha + 1} \lambda_n(\mu_\a, e^{i\pi/2})  = 2^{\a+1}L_\a.
    \end{equation}
This result will serve as our model case in the proof of Theorem \ref{thmain}.

\sect{Lemniscates}
In this section, we prove Theorem \ref{thmain} for lemniscates.

Let \( \sigma = \{ z \in \mathbb{C}: |T_N(z)| = 1\} \) be the level line of a polynomial \( T_N \), and assume that \( \sigma \) has no self-intersections. Let \( \deg(T_N) = N \).

The normal derivative of the Green's function with pole at infinity
of the outer domain to $\s$ at a point
$z\in \s$ is (see \cite[(2.2)]{Totiktrans}) $|T_N'(z)|/N$,
and since this normal derivative is $2\pi$-times the equilibrium density
of $\s$ (see \cite[II.(4.1)]{Nevanlinna}
or \cite[Theorem IV.2.3]{SaffTotik} and \cite[(I.4.8)]{SaffTotik}), it follows that
the equilibrium density on $\s$ has the form
\be \o_\s(z)=\frac{|T_N'(z)|}{2\pi N}.\label{so}\ee

If $z\in \s$, then there are $n$ points $z_1,\ldots,z_n\in \s$
with the property $T_N(z)=T_n(z_k)$, and for them (see \cite[(2.12)]{Totiktrans})
\begin{equation}\label{lemniscate_integral_2}
    \int_{\sigma} \Big( \sum_{i=1}^{N} f(z_i) \Big) |T_{N}^{'}(z)| ds_\sigma(z) = N \int_\sigma f(z)|T_{N}^{'}(z)| ds_\sigma(z).
\end{equation}
Furthermore, if \( g:\mathbb{T} \to \mathbb{C} \) is arbitrary, then
 (see \cite[(2.14)]{Totiktrans})
\begin{equation}\label{lemniscate_integral_3}
    \int_\sigma g(T_N(z)) |T_{N}^{'}(z)| ds_\sigma(z) = N \int_{0}^{2\pi} g(e^{it}) dt.
\end{equation}

Let $z_0\in \s$ be arbitrary, and define the measure
\begin{equation}\label{lemniscate_measure}
    d\mu_{\sigma}(z) = |z - z_0|^\alpha ds_{\sigma}(z), \quad \alpha > -1,
\end{equation}
where $s_\s$ denotes the arc measure on $\s$.
Without loss of generality we may assume that \( T_N(z_0) = e^{i\pi/2} \).
Our plan is to compare the Chritoffel functions for
the measure
$ \mu_\sigma $ with that for the measure $ \mu_\a $ which is supported on the unit circle and is defined via
\begin{equation}\label{lemniscate_mu_t}
    d\mu_\a(e^{it}) = |e^{it} - e^{i\pi/2}|^{\alpha} ds_{\mathbb{T}}(e^{it}),
\end{equation}
and for which the asymptotics of the Christoffel function was calculated in
(\ref{unit_circle_result}).

We shall prove that
\begin{equation}\label{lemniscate_result}
    \lim_{n \to \infty} n^{\alpha + 1} \lambda_n(\mu_\sigma, z_0) = \frac{L_\a}{(\pi \omega_\s(z_0))^{\alpha + 1}}
\end{equation}
where $L_\a$ is taken from (\ref{lalpha}).

\subsection{The upper estimate} Let $ \eta > 0 $ be an arbitrary small number, and select a $ \delta > 0 $ such that for every $ z $ with $ |z-z_0| < \delta $, we have
\begin{equation}\label{lemniscate_upper_estimate_separation}
\begin{aligned}
  \frac1{1+\eta}  |T_{N}^{'}(z_0)| & \leq  |T_{N}^{'}(z)| \le (1+\eta)  |T_{N}^{'}(z_0)|\\
 \frac{1}{1+\eta}  |T_{N}^{'}(z_0)||z-z_0| & \leq |T_N(z) - T_N(z_0)|
 \le (1+\eta)  |T_{N}^{'}(z_0)||z-z_0|
\end{aligned}
\end{equation}
(note that $T_N'(z_0)\not= 0$ because $\s$ has no self-intersections).
Let $ Q_n $ be the extremal polynomial for $ \lambda_n(\mu_\a , e^{i\pi/2}) $, where $ \mu_\a  $ is from (\ref{lemniscate_mu_t}). Define $ R_n $ as
\[
    R_n(z) = Q_n(T_N(z)) S_{n,z_0,L}(z),
\]
where $ S_{n, z_0, L} $ is the fast decreasing polynomial given by Corollary \ref{fdp_beta}
 for  the lemniscate set $L$ enclosed by $ \sigma $
(and for any fixed $0<\tau <1$ in Corollary \ref{fdp_beta}).
Note that $ R_n $ is a polynomial of degree $ nN + o(n) $ with $ R_n(z_0) = 1 $. Since $ S_{n, z_0, L} $ is fast decreasing, we have
\[
    \sup_{z \in L \setminus \{ z: |z-z_0| < \delta \}} |S_{n,z_0,L}(z)| = O(q^{n^{\tau_0}})
\]
for some $ q < 1 $ and $  \tau_0>0 $. The  Nikolskii-type inequality
in Lemma \ref{Niklemma} when applied to two subarcs of $\mathbb{T}$
which contain the upper resp. lower part of the unit circle, yields
\[
    \| Q_n \|_{\mathbb{T}} \leq C n^{(1+|\alpha|)/2} \| Q_n \|_{L^2(\mu_\a )} \leq Cn^{(1+|\a|)/2}.
\]
Therefore,
\[
    \sup_{z \in L \setminus \{ z: |z-z_0| < \delta \}} |R_n(z)| = O(q^{n^{\tau_0 / 2}}).
\]
It follows that
\begin{equation}\label{lemniscate_upper_1}
    \int_{|z-z_0| \geq \delta} |R_n(z)|^2 |z-z_0|^\alpha ds_{\sigma}(z) = O(q^{n^{\tau_0 /2}}).
\end{equation}
Using (\ref{lemniscate_upper_estimate_separation}), we have
\begin{align*}
    \int_{|z-z_0| < \delta} |R_n(z)|^2 & |z-z_0|^\alpha ds_\sigma(z) \\
    & \leq \int_{|z-z_0| < \delta} |Q_n(T_N(z))|^2 |z-z_0|^\alpha ds_\sigma(z) \\
    & \leq \frac{(1+\eta)^{|\alpha| + 1}}{|T_{N}^{'}(z_0)|^{\alpha + 1}} \int_{|z-z_0| < \delta} |Q_n(T_N(z))|^2 |T_N(z) - T_N(z_0)|^\alpha |T_{N}^{'}(z)| ds_\sigma(z) \\
    & \leq \frac{(1+\eta)^{|\alpha| + 1}}{|T_{N}^{'}(z_0)|^{\alpha + 1}} \int_{0}^{2\pi} |Q_n(e^{it})|^2 |e^{it} - e^{i\pi/2}|^\alpha dt \\
    & = (1+\eta)^{|\alpha| + 1} \frac{\lambda_n(\mu_\a , e^{i\pi/2})}{|T_{N}^{'}(z_0)|^{\alpha + 1}}.
\end{align*}
This and (\ref{lemniscate_upper_1}) imply
\[
    \lambda_{\deg(R_n)}(\mu_\sigma, z_0) \leq (1+\eta)^{|\alpha| + 1} \frac{\lambda_n(\mu_\a , e^{i\pi/2})}{|T_{N}^{'}(e^{i\pi/2})|^{\alpha + 1}} +O(q^{n^{\tau_0/2}}),
\]
from which
\begin{align*}
    \limsup_{n \to \infty} \deg(R_n)^{\alpha + 1} & \lambda_{\deg(R_n)}(\mu_\sigma, z_0) \\
    & \leq \limsup_{n \to \infty} (nN + o(n))^{\alpha + 1} (1+\eta)^{|\alpha| + 1} \frac{\lambda_n(\mu_\a , e^{i\pi/2})}{|T_{N}^{'}(z_0)|^{\alpha + 1}} \\
    & = (1+\eta)^{|\alpha| + 1} \frac{N^{\alpha + 1}}{|T_{N}^{'}(z_0)|^{\alpha + 1}} 2^{\alpha + 1}L_\a,
\end{align*}
where we used (\ref{unit_circle_result}).
Since $ \eta > 0 $ is arbitrary, we obtain from (\ref{so}) (use also (\ref{limsup}))
\be
    \limsup_{n \to \infty} n^{\alpha + 1} \lambda_n(\mu_\sigma, z_0) \leq \frac{N^{\alpha + 1}}{|T_{N}^{'}(z_0)|^{\alpha + 1}} 2^{\alpha + 1}L_\a
=  \frac{L_\a}{(\pi \omega_\s(z_0))^{\alpha + 1}} .\label{eer}\ee

\subsection{The lower estimate } Let $ P_n $ be the extremal polynomial for $ \lambda_n(\mu_\sigma, z_0) $, and let $ S_{n,z_0,L} $ be the fast decreasing polynomial given by Corollary \ref{fdp_beta}
for the closed lemniscate domain $L$
enclosed by $\s$ (with some fixed $\tau<1$).
As before, we obtain from Lemma \ref{Niklemma}
\be \|P_n\|_\s=O(n^{(1+|\a|)/2}).\label{pmn}\ee
Define $ R_n(z) = P_n(z) S_{n,z_0,L}(z) $. $ R_n $ is a polynomial of degree $ n + o(n) $ and $ R_n(z_0) = 1 $. Similarly to the previous section, we have
\begin{equation}\label{lemniscate_lower_rn_sup}
    \sup_{z \in L \setminus \{ z: |z-z_0| < \delta \}} |R_n(z)| = O(q^{n^{\tau_0/2}})
\end{equation}
for some $ q < 1 $ and $  \tau_0 >0$. Since the expression $ \sum_{k=1}^{N} R_n(z_k) $, where $ \{ z_1, \dots, z_N \} = T_{N}^{-1}(T_N(z)) $, is symmetric in the variables
$ z_k $, it is a sum of their elementary symmetric polynomials. For more details on this idea, see \cite{T2}.
Therefore,
there is a polynomial $ Q_n $ of degree at most $ \deg(R_n)/N = (n + o(n))/N $ such that
\[
    Q_n(T_N(z)) = \sum_{k=1}^{N} R_n(z_k), \qquad z \in \sigma.
\]

We claim that for every $ z \in \sigma$, we have
\begin{equation}\label{lemniscate_lower_qn}
    |Q_n(T_N(z))|^2 \leq \sum_{k=1}^{N} |R_n(z_k)|^2 + O(q^{n^{\tau_0/2}}).
\end{equation}
Indeed, since $ \sigma $ has no self intersection, $ |z_k - z_l| $ cannot be arbitrarily small for distinct $ k $ and $ l $. As a consequence, for every $z$ at most
one $z_j$ belongs to the set $\{z\sep |z-z_0|<\d\}$ if $\d$ is sufficiently small,
and hence, in the sum
\[
    |Q_n(T_N(z))|^2 \leq \sum_{k=1}^{N} \sum_{l=1}^{N} |R_n(z_k)||R_n(z_l)|,
\]
every term with $ k \neq l $ is $ O(q^{n^{\tau_0/2}}) $ (use (\ref{pmn})
and (\ref{lemniscate_lower_rn_sup})).

Now let $ \delta > 0 $ be so small that for every $ z $ with $ |z-z_0| < \delta $
the inequalities in (\ref{lemniscate_upper_estimate_separation}) hold.
 Then (\ref{lemniscate_integral_2}) and (\ref{lemniscate_lower_qn}) give
 (note that $T_N(z)=T_N(z_k)$ for all $k$)
\begin{align*}
    \int_\sigma |Q_n(T_N(z))|^2 & |T_{N}^{'}(z)| |T_N(z) - T_N(z_0)|^\alpha ds_\sigma(z) \\
    & \leq O(q^{n^{\tau_0/2}}) + \int_\sigma \left(\sum_{k=1}^{N} |R_n(z_k)|^2\right) |T_{N}^{'}(z)| |T_N(z) - T_N(z_0)|^{\alpha} ds_\sigma(z) \\
    & = O(q^{n^{\tau_0/2}}) + \int_\sigma \left(\sum_{k=1}^{N} |R_n(z_k)|^2
    |T_N(z_k) - T_N(z_0)|^{\alpha}\right) |T_{N}^{'}(z)|  ds_\sigma(z) \\
    & = O(q^{n^{\tau_0/2}})  + N \int_\sigma |R_n(z)|^2|T_N(z) - T_N(z_0)|^\alpha  |T_{N}^{'}(z)| ds_\sigma(z) \\
    & \leq O(q^{n^{\tau_0/2}}) + (1+\eta)^{|\alpha| + 1} |T_{N}^{'}(z_0)|^{\alpha + 1} N \int_{|z-z_0| < \delta} |P_n(z)|^2 |z - z_0|^{\alpha} ds_\sigma \\
    & \leq O(q^{n^{\tau_0/2}}) + (1+\eta)^{|\alpha| + 1} |T_{N}^{'}(z_0)|^{\alpha + 1} N \lambda_n(\mu_\sigma, z_0).
\end{align*}
Since $ Q_n(T_N(z_0)) = 1 + o(1) $, we get from (\ref{lemniscate_integral_3})
\begin{align*}
    \int_\sigma |Q_n(T_N(z))|^2 & |T_{N}^{'}(z)| |T_N(z) - T_N(z_0)|^{\alpha} ds_\sigma(z) \\
    & = N \int_{0}^{2\pi} |Q_n(e^{it})|^2|e^{it} - e^{i\pi/2}|^{\alpha} dt \\
    & \geq (1+o(1)) N \lambda_{\deg(Q_n)} (\mu_\a , e^{i\pi/2}).
\end{align*}
Hence, the inequality
\[
    (1+o(1)) \lambda_{\deg(Q_n)} (\mu_\a , e^{i\pi/2}) \leq
    O(q^{n^{\tau_0/2}}) + (1+\eta)^{|\alpha| + 1} |T_{N}^{'}(z_0)|^{\alpha + 1} \lambda_n(\mu_\sigma, z_0)
\]
holds. Using  that $ \deg(Q_n) \leq (n + o(n))/N $, we can conclude
\begin{align*}
    \liminf_{n \to \infty} \deg(Q_n)^{\alpha + 1} & \lambda_{\deg(Q_n)}(\mu_\a , e^{i\pi/2}) \\
    & \leq (1+\eta)^{|\alpha| + 1} |T_{N}^{'}(z_0)|^{\alpha + 1} \liminf_{n \to \infty} \Big( \frac{n + o(n)}{N} \Big)^{\alpha + 1} \lambda_n(\mu_\sigma, z_0) \\
    & \leq (1 + \eta)^{|\alpha| + 1} \frac{|T_{N}^{'}(z_0)|^{\alpha + 1}}{N^{\alpha + 1}} \liminf_{n \to \infty} n^{\alpha + 1} \lambda_n(\mu_\sigma, z_0).
\end{align*}
Since  $ \eta > 0 $ is arbitrary, we obtain
again from (\ref{unit_circle_result}) and (\ref{so})
\[
    \frac{L_\a}{(\pi \omega_\s(z_0))^{\alpha + 1}}  \leq \liminf_{n \to \infty} n^{\alpha + 1} \lambda_n(\mu_\sigma, z_0),
\]
which, along with (\ref{eer}), proves (\ref{lemniscate_result}).

\sect{Smooth Jordan curves}
In this section, we verify Theorem \ref{thmain} for a finite union
$\G$ of smooth Jordan curves and for a measure
\begin{equation}\label{general_mu_def}
    d\mu(z) = w(z) |z-z_0|^{\alpha} ds_\G(z),
\end{equation}
where $s_\G$ is the arc measure on $\G$.
Recall that a Jordan curve is a homeomorhic image
of a circle, while a Jordan arc is a homeomorhic image of a segment.
From the point of view of our technique there is a
big difference between arcs and curves, and in the
present section we shall only work with Jordan curves.

Let $ \Gamma $ be a finite system of $C^2$ Jordan curves lying exterior to each other
and let $ \mu $ be a measure on $ \Gamma $
given in (\ref{general_mu_def}),
where $ w $ is a continuous and strictly positive function. Our goal is to prove that
\begin{equation}\label{main_result}
    \lim_{n \to \infty} n^{\alpha + 1} \lambda_n(\mu, z_0) =  \frac{w(z_0)}{(\pi \omega_{\G}(z_0))^{\alpha + 1}}L_\a
\end{equation}
with $L_\a$ from (\ref{lalpha}).
We shall deduce this from the result for lemniscates proved in the
preceding section.

We will approximate $ \Gamma $ with lemniscates using the following theorem, which was proven  in \cite{NT}.
\begin{prop}\label{lemniscate_approximation}
    Let $ \Gamma $ consist of finitely many Jordan curves lying exterior to each other,
     let $ P \in \Gamma $, and assume that in a neighborhood of $ P $ the curve $\G$ is
     $ C^2 $-smooth. Then, for every $ \varepsilon > 0 $, there is a lemniscate $ \s=\sigma_P $ consisting of Jordan curves such that $ \sigma $ touches $ \Gamma $ at $ P $, $ \sigma $ contains $ \G $ in its interior except for the point $ P $, every component of $ \sigma $ contains
in its interior precisely one component of $ \Gamma $, and
\begin{equation}\label{lemniscate_theorem_exterior}
\o_\G(P)\le \o_\s(P)+\e.
\end{equation}

Also, for every $ \varepsilon > 0 $, there exists another
 lemniscate $ \sigma=\s_P $ consisting of Jordan curves such that $ \sigma $ touches $ \Gamma $ at $ P $, $ \sigma $ lies strictly inside $\G $ except for the point $ P $, $ \sigma $ has exactly one component lying inside every component of $ \Gamma $, and
\begin{equation}\label{lemniscate_theorem_interior}
\o_\s(P)\le \o_\G(P)+\e.
\end{equation}
\end{prop}
Of course, the phrase ``$\G$ lies inside $\s$" means that the components of
$\G$ lie inside (i.e. in the interior of) the corresponding components of $\s$.
See Figure \ref{figlemnisc}.

\begin{figure}
\centering
\includegraphics[scale=0.7]{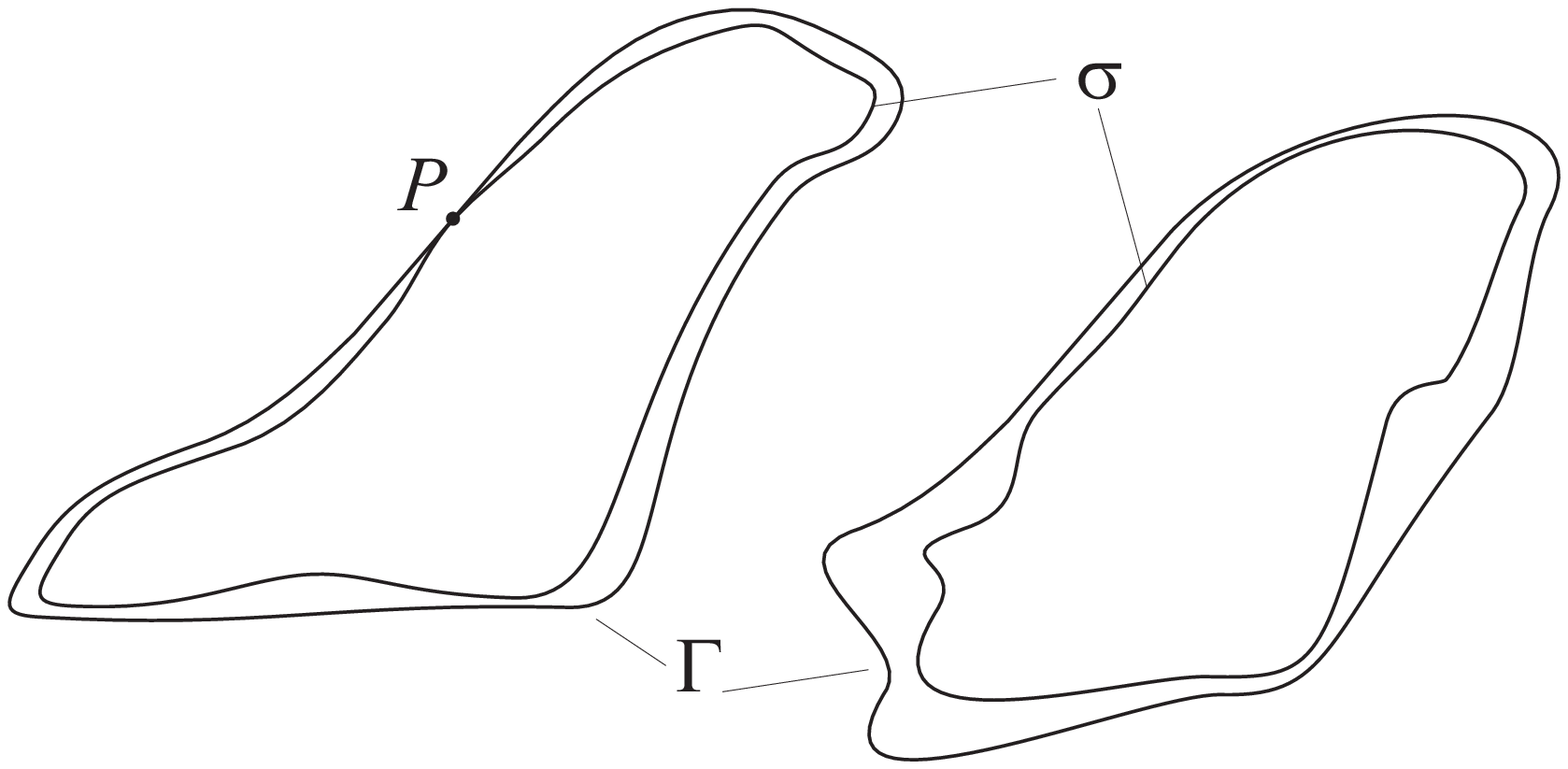}
\caption{\protect\label{figlemnisc} The $\G$ and the leminscate
$\s$ as in the second half of Proposition \ref{lemniscate_approximation}}
\end{figure}

Note that in (\ref{lemniscate_theorem_exterior}) the inequality
$\o_\s(P)\le \o_\G(P)$ is automatic since $\G$ lies inside $\s$. In a similar
way, in (\ref{lemniscate_theorem_interior}) the inequality
$\o_\G(P)\le \o_\s(P)$ holds.

Actually, in \cite{NT} the conditions (\ref{lemniscate_theorem_exterior})
and (\ref{lemniscate_theorem_interior}) were formulated in terms of the
normal derivatives of the Green's function of the outer domains to $\G$ and $\s$,
but, in view of the fact that this latter is just
$2\pi$-times the equilibrium density (see \cite[II.(4.1)]{Nevanlinna} or \cite[Theorem IV.2.3]{SaffTotik} and \cite[(I.4.8)]{SaffTotik}), the two formulations are
equivalent.

\subsection{The lower estimate} Let $ P_n $ be the extremal polynomial for $ \lambda_n(\mu, z_0) $, and for some $\tau>0$ let $ S_{\tau n, z_0, K} $ be the fast decreasing polynomial given by Proposition
\ref{fdp} with some $\g>1$ to be chosen below, where $ K $ is the set enclosed by $ \Gamma $.
Let $ \s=\s_{z_0} $ be a lemniscate inside $ \Gamma $ given by the second part of Proposition  \ref{lemniscate_approximation}, and suppose that $ \s  = \{ z: |T_N(z)| = 1 \} $, where $ T_N $ is a polynomial of degree $ N $ and $ T_N(z_0) = e^{i\pi/2} $. Define
$ R_n  = P_n S_{\tau n,z_0,K}$.
Note that $ R_n $ is a polynomial of degree at most $ (1+\tau)n $ and $ R_n(z_0) = 1 $.
These will be the test polynomials in estimating
the Christoffel function for the measure
\[d\mu_\s(z):=|z-z_0|^\a d s_\s(z)\]
 on $\s$, but first we need two nontrivial facts for these polynomials.

\begin{lem}\label{general_main_lemma} Let $\frac12<\b<1$ be fixed. For
$ z \in \G $ such that $ |z-z_0| \leq 2n^{-\beta} $, let $ z^* \in \s  $ be
the point such that $ s_\s([z_0,z^*]) = s_\Gamma([z_0,z]) $ holds
(actually, there are two such points, we choose as $z^*$ the one
the lies closer to $z$).
Then the mapping $ q(z) = z^* $ is one to one, $ |q(z) - z| \leq C |z-z_0|^2 $, $ ds_\G(z) = ds_{\s }(z^*) $, $|q'(z_0)| = 1$,
and with the notation $ I_n := \{ z^* \in \s : |z^* - z_0| \leq n^{-\beta} \} $, we have
\begin{equation}\label{general_integral_estimate}
    \Bigg| \int_{z^* \in I_n} |R_n(z^*)|^2 |z - z_0|^\alpha ds_{\s }(z^*) - \int_{z^* \in I_n} |R_n(z)|^2 |z-z_0|^\alpha ds_\G(z) \Bigg| = o(n^{-(1 + \alpha)}).
\end{equation}
\end{lem}

On the left-hand side $z=q^{-1}(z^*)$, so the integrand is a function of $z^*$.

\proof.
First of all we mention that $ |q'(z_0)| = 1 $, i.e. for every $ \varepsilon > 0 $, if $ |z-z_0| $ is small enough, then
\[
    1-\varepsilon \leq \frac{|q(z) - z_0|}{|z - z_0|} \leq 1 + \varepsilon,
\]
which is clear since $q(z)=z+O(|z-z_0|^2)$.

We proceed to prove (\ref{general_integral_estimate}).
\begin{align*}
    & \Bigg| \int_{z^* \in I_n} |R_n(z^*)|^2 |z-z_0|^\alpha ds_\s(z^*) - \int_{z^* \in I_n} |R_n(z)|^2 |z-z_0|^\alpha ds_\G(z)  \Bigg| \\
    & \leq \Bigg| \int_{z^* \in I_n} \Big( |R_n(z^*)|^2 - |R_n(z)|^2 \Big) |z-z_0|^{\alpha} ds_\G(z) \Bigg| \\
    & \leq \int_{z^* \in I_n} \Big| |R_n(z^*)|^2 - |R_n(z)|^2 \Big| |z-z_0|^{\alpha} ds_\G(z) :=A.
    \end{align*}
    Using the H\"older and Minkowski inequalities we can continue as
\bea
   A  & \leq& \Bigg( \int_{z^* \in I_n} |R_n(z^*) - R_n(z)|^2 |z-z_0|^{\alpha} ds_\G(z) \Bigg)^{1/2} \times \label{uv}\\
    & &\Bigg\{ \Bigg( \int_{z^* \in I_n} |R_n(z^*)|^2 |z-z_0|^\alpha ds_\G(z) \Bigg)^{1/2}  + \Bigg( \int_{z^* \in I_n} |R_n(z)|^2 |z-z_0|^\alpha ds_\G(z) \Bigg)^{1/2}  \Bigg\}.\nonumber \eea
We estimate these integrals term by term.

$ P_n $ is extremal for $ \lambda_n(\mu, z_0) = O(n^{-(\a+1 )}) $ (see Lemma \ref{a_priori_estimate}),
 therefore we have (use also that $|R_n(z)|\le |P_n(z)|$)
\be
     \Bigg( \int_{z^* \in I_n} |R_n(z)|^2 |z-z_0|^\alpha ds_\G(z) \Bigg)^{1/2} \leq
     C n^{-\frac{\a+1 }{2}}.
\label{1term}\ee
This takes care of the third term in (\ref{uv}).

The estimates for the other two terms
 differ in the cases $ \alpha \geq 0 $ and $ \alpha < 0 $.

Assume first that $ \alpha \geq 0 $.
From Lemma \ref{Niklemma}, we get for any closed subarc $J_1\subset J$
\[ \| R_n \|_{J_1} \leq C n^{(\a+1)/2} \| R_n \|_{L^2(\mu)} \leq C,\]
where we used Lemma \ref{a_priori_estimate} and $|R_n(z)|\le |P_n(z)|$.
Choose this  $J_1$
so that it contains $z_0$ in its interior. Next,
note that if $z^*\in I_n$,
then $|z^*-z|\le Cn^{-2\b}$, so dist$(z^*,z)\le C/n$.
Therefore,
an application of Lemma \ref{bernstein} yields for such $z$
\[
    \frac{|R_n(q(z)) - R_n(z)|}{|q(z) - z|} \leq C n \| R_n \|_{J_1},
\]
and so
\be
    |R_n(q(z)) - R_n(z)| \leq C n|q(z) - z| \leq C n^{1 - 2 \beta}.
\label{viva}\ee
Since $ s_{\s }(I_n) \leq C n^{-\beta} $ is also true, we have (recall that $z^*=q(z)$)
\bean
    \Bigg( \int_{z^* \in I_n} |R_n(z^*) - R_n(z)|^2 |z-z_0|^{\alpha} ds_\G(z) \Bigg)^{1/2} &\leq & C\Big(  n^{-\beta} n^{2-4\beta} n^{-\alpha \beta} \Big)^{1/2}\\
& =& C n^{1 - \frac{5+\alpha}{2}\beta}.\eean
This is the required estimate for the first term in (\ref{uv}).

Finally, for the middle term in (\ref{uv}), we have
\begin{align*}
    \Bigg( \int_{z^* \in I_n} |R_n(z^*)|^2 & |z-z_0|^\alpha ds_\G(z) \Bigg)^{1/2} \\
    & = \Bigg( \int_{z^* \in I_n} \Big| |R_n(z^*)|^2 - |R_n(z)|^2 + |R_n(z)|^2 \Big| |z-z_0|^{\alpha} ds_\G(z) \Bigg)^{1/2} \\
    & \leq \Bigg( \int_{z^* \in I_n} \Big| |R_n(z^*)|^2 - |R_n(z)|^2 \Big| |z-z_0|^{\alpha} ds_\G(z) \Bigg)^{1/2} \\
    & \phantom{aaaaaaaaaaaaa} + \Bigg( \int_{z^* \in I_n} |R_n(z)|^2 |z-z_0|^\alpha ds_\G(z) \Bigg)^{1/2} \\
    & \leq A^{1/2} + C n^{-\frac{\a+1 }{2}},
\end{align*}
where $A$ is the left-hand side in (\ref{uv}), and where we also used
(\ref{1term}).

 Combining these we get
\begin{align*}
    A & \leq C n^{1 - \frac{5+\alpha}{2}\beta} \Big( A^{1/2} + C n^{-\frac{\a+1 }{2}} \Big)
     \leq C A^{1/2} n^{1 - \frac{5+\alpha}{2}\beta} + C n^{\frac12 - \frac{\alpha}{2} - \frac{5 + \alpha}{2}\beta} \\
    & \leq C \max\{ A^{1/2} n^{1 - \frac{5+\alpha}{2}\beta}, n^{\frac12 - \frac{\alpha}{2} - \frac{5 + \alpha}{2}\beta} \}.
\end{align*}
Therefore $ A \leq C n^{2-(5+\alpha)\beta} $ or $ A \leq Cn^{\frac12 - \frac{\alpha}{2} - \frac{5 + \alpha}{2}\beta} $. If $ \beta <1$ is sufficiently close to $ 1 $, then both imply $ A = o(n^{-(\a+1 )}) $.

Now assume that $ \alpha < 0 $. From Lemma \ref{Niklemma}, we get for any closed subarc $J_1\subset J$
\[ \| R_n \|_{J_1} \le \| P_n \|_{J_1} \leq C n^{1/2} \| P_n \|_{L^2(\mu)} \leq C n^{-\alpha/2},\]
and we may assume that here $J_1$ is such that it contains
a neighborhood of $z_0$. Therefore, in this case (\ref{viva}) takes the form
\[
    |R_n(z^*) - R_n(z)| \leq C n^{1-\alpha/2 - 2 \beta}.
\]
Since
\[ \int_{z^* \in I_n} |z-z_0|^{\alpha} ds_\G(z) \leq C n^{-\alpha \beta - \beta} ,\]
 we obtain
\begin{align*}
    \Bigg( \int_{z^* \in I_n} |R_n(z^*) - R_n(z)|^2 & |z-z_0|^{\alpha} ds_\G(z) \Bigg)^{1/2} \leq C n^{1-\frac{\alpha}{2}-2\beta-\frac{(\alpha + 1) }{2}\beta},
\end{align*}
which is the required estimate for the first term in (\ref{uv}).
Finally, for the middle term in (\ref{uv}) we get, similarly as before,
\begin{align*}
    \Bigg( \int_{z^* \in I_n} |R_n(z^*)|^2 & |z-z_0|^\alpha ds_\G(z) \Bigg)^{1/2} \leq A^{1/2}  + C n^{-\frac{\a+1 }{2}}.
\end{align*}
As previously, we can conclude from these
\[A\le Cn^{1-\frac{\a}{2}-2\b-\frac{\a+1}{2}\b}(A^{1/2}+n^{-\frac{\a+1}{2}}),\]
which implies
\[
    A \leq C \max\{ n^{2 - \alpha - 4\beta - (\alpha + 1)\beta}, n^{\frac12
     -\a - 2\beta - \frac{\a+1 }{2}\beta} \}.
\]
If $ \beta $ is sufficiently close to $ 1 $, then this yields again
 $ A = o(n^{-(\a+1 )}) $, as needed.
\endproof

In what follows we keep the notations from the preceding proof.
In the following lemma let
$\D_\d(z_0)=\{z\sep |z-z_0|\le \d\}$ be the disk
about $z_0$ of radius $\d$.

Note that up to this point the $\g>1$ in Proposition \ref{fdp} was arbitrary. Now we
specify how close it should be to 1.
\begin{lem}\label{lemmaadded} If $0<\b<1$ is fixed and $\g>1$ is chosen
so that $\b\g<1$, then
\be \|R_n\|_{K\setm \D_{n^{-\b}/2}}(z_0)=o(n^{-1-\a}).\label{hhhfff}\ee
\end{lem}
Recall that here $K$ is the set enclosed by $\G$.

\proof. Let us fix a $\d>0$ such that the intersection
$\G\cap \D_\d(z_0)$  lies in the interior of the arc
$J$ from Theorem \ref{thmain}.
By $\mu\in {\bf Reg}$ and the trivial estimate
$\|P_n\|_{L^2(\mu)}=O(1)$ we get that no matter how small $\e>0$
is given, for sufficiently large $n$ we have
$\|P_n\|_\G\le (1+\e)^n$. On the other hand,
in view of Proposition \ref{fdp}, we have for $z\not\in \D_\d(z_0)$, $z\in K$,
\[|S_{\tau n,z_0,K}(z)|\le C_\g e^{-c_\g\tau n\d^2},\]
so
\be \|R_n\|_{K\setm \D_\d}(z_0)=o(n^{-1-\a})\label{nnn}\ee
 certain holds.

 Consider now $K\cap \D_\d(z_0)$. Its boundary consists of
 the arc $\G\cap \D_\d(z_0)$, which is part of $J$, and
 of an arc on the boundary of $\D_\d(z_0)$, where we already
 know the  bound (\ref{nnn}). On the other hand,
 on $\G\cap \D_\d(z_0)$ we have, by Lemma \ref{Niklemma},
 \[|P_n(z)|\le Cn^{(1+|\a|)/2}\|P_n\|_{L^2(\mu)}\le Cn^{(1+|\a|)/2}.\]
Therefore, by the maximum principle, we obtain
the same bound (for large $n$) on the whole
set $K\cap \D_\d(z_0)$. As a consequence,
for $z\in K\setm \D_{n^{-\b}/2}$
 \[|R_n(z)|\le Cn^{(1+|\a|)/2}e^{-c_\g\tau n (n^{-\b}/2)^\g}=o(n^{-1-\a})\]
if we choose $\g>1$ in Proposition \ref{fdp} so that $\b \g<1$.
These prove (\ref{hhhfff}).\endproof

After these preliminaries we return to the proof of Theorem
\ref{thmain}, more precisely to the lower  estimate
of $\l_n(\mu,z_0)$.

Let $ \eta > 0 $ be arbitrary, and let $ n $ be so large that
\[
    \frac{1}{1+ \eta} w(z_0) \leq w(z) \leq (1+\eta) w(z_0), \qquad
  \frac{1}{1+\eta}  |z-z_0|\le    |q(z) - z_0| \leq (1+\eta) |z-z_0|\]
hold for all $ z^* \in I_n $, where $I_n$ is the set
from Lemma \ref{general_main_lemma}. Then we obtain from Lemma \ref{general_main_lemma}
(recall
that $z^*=q(z)$)
\begin{align*}
    \int_{z^* \in I_n} |R_n(z^*)|^2 & |z^* - z_0|^{\alpha} ds_{\s }(z^*) \\
    & \leq (1+\eta)^{|\a|} \int_{z^* \in I_n} |R_n(z^*)|^2 |z - z_0|^\alpha ds_\G(z) \\
    & \leq (1+\eta)^{|\a|} \int_{z^* \in I_n} |R_n(z)|^2 |z-z_0|^\alpha ds_\G(z) + o(n^{-(\a+1 )}) \\
    & \leq \frac{(1+ \eta)^{|\a|+1}}{w(z_0)} \int_{z^* \in I_n} |R_n(z)|^2 w(z) |z-z_0|^\alpha ds_\G(z) + o(n^{-(\a+1 )}) \\
    & \leq \frac{(1+\eta)^{|\a|+1}}{w(z_0)} \lambda_n(\mu, z_0) + o(n^{-(\a+1 )}).
\end{align*}
On the other hand, if we notice that if, for some $z\in \s$,
we have $z^*\not\in I_n$ then necessarily   $|z-z_0|\ge n^{-\b}/2$,
we obtain from Lemma \ref{lemmaadded}
\begin{align*}
    \int_{z^* \in \s  \setminus I_n} |R_n(z^*)|^2 & |z^*-z_0|^{\alpha} ds_{\s }(z^*) =o(n^{-(1+\a)}).
\end{align*}
Combining these, it follows that
\bean
    \lambda_{\deg(R_n)}(\mu_{\s }, z_0)  &\leq& \int_{z \in \s } |R_n(z^*)|^2 |z^* - z_0|^\alpha ds_\s(z^*) \\
    & \leq &\frac{(1+\eta)^{|\a|+1}}{w(z_0)} \lambda_n(\mu, z_0)+ o(n^{-(\a+1 )}) .
\eean
Since $ \deg(R_n) \le (1+\tau)n $, we can conclude from
(\ref{lemniscate_result}) (see also (\ref{liminf}))
\bean
    \frac{L_\a}{(\pi\o_\s(z_0))^{\alpha + 1}}
    & = &\liminf_{n \to \infty} \deg(R_n)^{\alpha + 1} \lambda_{\deg(R_n)}(\mu_{\s }, z_0) \\
    & \leq& \liminf_{n \to \infty} (1+\tau)^{\a +1}\frac{(1+\eta)^{|\a|+1}}{w(z_0)} n^{\alpha + 1} \lambda_n(\mu, z_0).
\eean
But here $ \tau, \eta > 0 $  are arbitrary, so we get
\begin{align*}
    \liminf_{n \to \infty} n^{\alpha + 1} \lambda_n(\mu, z_0) \ge \frac{w(z_0)}{(\pi \omega_{\s}(z_0))^{\alpha + 1}} L_\a.
\end{align*}
As $\o_\s(z_0)\le \o_\G(z_0)+\e$ (see (\ref{lemniscate_theorem_interior})),
for $\e\to 0$ we finally arrive at the lower estimate
\be \liminf_{n \to \infty} n^{\alpha + 1} \lambda_n(\mu, z_0) \ge \frac{w(z_0)}{(\pi \omega_{\G}(z_0))^{\alpha + 1}}L_\a.
\label{aaa}\ee

\subsection{The upper estimate} Let now $ \s  $ be the lemniscate given by the first part of Proposition
\ref{lemniscate_approximation}, and let $ P_n $ be the polynomial extremal for $ \lambda_n(\mu_{\s }, z_0) $. Define, with some $\tau>0$,
\[
    R_n(z) = P_n(z) S_{\tau n, z_0, L}(z),
\]
where $ S_{\tau n, z_0, L} $ is the  fast decreasing polynomial given by
Proposition \ref{fdp} for the lemniscate set $ L $ enclosed by $\s$
(with some $\g>1$). Let $ \eta > 0 $ be arbitrary, $\frac12<\b<1$ as before,
 and suppose that $ n $ is so large such that
\begin{align*}
    \frac{1}{1+\eta} w(z_0) & \leq w(z) \leq (1+\eta) w(z_0) \\
    \frac{1}{\eta+1 }  & \leq |q'(z)| \leq (1+\eta)\\
    \frac{1}{1+\eta} |z-z_0| & \leq |q(z) - z_0| \leq (1+\eta) |z-z_0|
\end{align*}
are true for all $ |z-z_0| \leq n^{-\beta} $. Using Lemma \ref{general_main_lemma}
(more precisely its version when  $\s$ encloses $\G$) we have
(recall again that $z^*=q(z)$)
\begin{align*}
    \int_{z^* \in I_n} |R_n(z)|^2 & w(z) |z-z_0|^\alpha ds_\G(z) \\
    & \leq (1+\eta) w(z_0) \int_{z^* \in I_n} |R_n(z)|^2 |z-z_0|^\alpha ds_\G(z) \\
    & \leq (1+\eta) w(z_0) \int_{z^{*} \in I_n} |R_n(z^*)|^2 |z-z_0|^{\alpha} ds_{\s }(z^*) + o(n^{-(\a+1 )}) \\
    & \leq (1+\eta)^{|\a|+1} w(z_0) \int_{z^* \in I_n} |R_n(z^*)|^2 |z^* - z_0|^{\alpha} ds_{\s }(z^*) + o(n^{-(\a+1 )}) \\
    & \leq (1+\eta)^{|\a|+1} w(z_0) \lambda_n(\mu_{\s },z_0) + o(n^{-(\a+1 })).
\end{align*}

On the other hand, Lemma \ref{lemmaadded} (but now
applied for the system of curves $\s$ rather than for $\G$) implies, as before,
\[
    \int_{\Gamma \setminus \D_{n^{-\b}/2}(z_0)} |R_n(z)|^2 |z - z_0|^\a d\mu(z) = o(n^{-(1+\a)}).
\]
Therefore,
\[
    \lambda_{\deg(R_n)}(\mu, z_0) \leq (1+\eta)^{|\a|+1} w(z_0) \lambda_n(\mu_{\s }, z_0) + o(n^{-(\a+1 )}),
\]
which, similarly to the lower estimate, upon using (\ref{lemniscate_result}) and letting $\tau,\eta$ tend to zero, implies  (see also
(\ref{limsup}))
\begin{align*}
    \limsup_{n \to \infty} n^{\alpha + 1} \lambda_n(\mu, z_0) & \leq  \frac{w(z_0)}
    {(\pi\o_\s(z_0))^{\a+1}} L_\a.
\end{align*}
Here, in view of (\ref{lemniscate_theorem_exterior}),
$\o_\G(z_0)\le \o_\s(z_0)+\e$, hence for $\e\to 0$ we
conclude
\begin{align*}
    \limsup_{n \to \infty} n^{\alpha + 1} \lambda_n(\mu, z_0) & \leq \frac{w(z_0) }
    {(\pi\o_\G(z_0))^{\a+1}}L_\a.
\end{align*}
This and (\ref{aaa}) prove (\ref{main_result}).\endproof

\sect{Piecewise smooth Jordan curves}
The proof in the preceding section can be carried out
without any difficulty if $\G$ consists of piecewise
$C^2$-smooth Jordan curves, provided that  in a neighborhood of
$z_0$ the $\G$ is $C^2$-smooth. Indeed, in that case
we can still talk about $\o_\G$ which is continuous
where $\G$ is $C^2$-smooth (see \cite[Proposition 2.2]{Totikarc}), and in the above
proof the $C^2$-smoothness was used only in
a neighborhood of $z_0$. Therefore,
we have

\begin{prop}\label{thmainprop}
      Let $ \Gamma $ consist of finitely many
      disjoint, piecewise $C^2$-smooth Jordan curves. Let $z_0\in\G$,
      and in a neighborhood of
      $ z_0\in \G$ let $\G$ be $C^2$-smooth.  Then, for the measure
      $\mu$ given in (\ref{general_mu_def}), we have (\ref{main_result}).
\end{prop}

\sect{Arc components}
In this section, we prove Theorem \ref{thmain} when
$\G$ is a union of $C^2$-smooth Jordan curves and arcs, and
$\mu$ is the measure  (\ref{general_mu_def})
considered before. To be more specific, our aim
is to verify
\begin{prop}\label{thmainprop*}
      Let $ \Gamma $ consist of finitely many
      disjoint $C^2$-smooth Jordan curves or arcs lying exterior
      to each other, and let $z_0\in\G$.
     Assume that in a neighborhood of
      the point $ z_0\in \G$, the $\G$ is $C^2$-smooth, and $z_0$ is not
      an endpoint of an arc component of $\G$.  Then, for the measure
      (\ref{general_mu_def})  where $w$ is continuous and positive and $\a>-1$, we have
(\ref{main_theorem_equation}).
\end{prop}

 We shall
need some facts about Bessel functions, and a
discretization of the equilibrium measure $\nu_\G$ that uses
the zeros of an appropriate Bessel function.

\subsection{Bessel functions and some local asymptotics}
We shall need the Bessel function of the first kind of order $\b>0$:
\[J_\b(z)=\sum_{n=0}^\i\frac{(-1)^n (z/2)^{2n+\b}}{n!\G(n+\b+1)},\]
as well as the functions (c.f. \cite{Lubinsky2})
\[\mathbb{J}_\b(u,v)=\frac{J_\b(\sqrt u)\sqrt vJ_\b'(\sqrt v)-
J_\b(\sqrt v)\sqrt uJ_\b'(\sqrt u)}{2(u-v)},\]
\[\mathbb{J}_\b^*(z)=\frac{{J}_\b(z)}{z^{\b}},
\qquad \mathbb{J}_\b^*(u,v)=\frac{\mathbb{J}_\b(u,v)}{u^{\b/2}v^{\b/2}}.\]
These latter ones are analytic, and we have
\[\mathbb{J}_\b^*(u,0)=\frac{1}{2^{2\b+1}u}\sum_{n=1}^\i
\frac{(-1)^n(\sqrt u/2)^{2n}}{n!\G(n+\b+1)}\left(\frac{\b}{\G(\b+1)}-\frac{2n+\b}{\G(\b+1)}\right)
=\frac{\mathbb{J}^*_{\b+1}(\sqrt u)}{2^{\b+1}\G(\b+1)}.\]

Let $d\nu_0(x)$ be the measure $x^\b dx$ with support $[0,2]$, and
$K_n^{(0)}(x,t)$ its $n$-th reproducing kernel. It is known
(see \cite[(1.2)]{Lubinsky1} or \cite[(4.5.8), p. 72]{Szego}) that
\[\frac{K_n^{(0)}\left(\frac{x^2}{2n^2},0\right)}{K_n^{(0)}(0,0)}
=(1+o(1)) \frac{\mathbb{J}_\b^*(x^2,0)}
{\mathbb{J}_\b^*(0,0)},\]
which  holds
uniformly for  $|x|\le A$ with any fixed $A$. We have already mentioned
(see e.g.
\cite[Theorem 3.1.3]{Szego}) that the polynomial
$K_n^{(0)}(t,0)/K_n^{(0)}(0,0)$ is the extremal polynomial of degree $n$
for $\l_n(\nu_0,0)$, so the preceding relation
gives an asymptotic formula for this extremal polynomial
on intervals $[0,A/n^2]$. If now $d\nu_1(x)=(2x)^\b dx$ but with support
$[0,1]$, and $K_n^{(1)}$ is the associated reproducing kernel,
then $K_n^{(1)}(t,0)/K_n^{(1)}(0,0)$ is the extremal polynomial
of degree $n$ for $\l_n(\nu_1,0)$, and it is clear that
this is just a scaled version of the extremal polynomial for $\nu_0$:
\[\frac{K_n^{(1)}(t,0)}{K_n^{(1)}(0,0)}=\frac{K_n^{(0)}(2t,0)}{K_n^{(0)}(0,0)}.\]
Therefore,
\[\frac{K_n^{(1)}\left(\frac{x^2}{4n^2},0\right)}{K_n^{(1)}(0,0)}=(1+o(1)) \frac{\mathbb{J}_\b^*(x^2,0)}
{\mathbb{J}_\b^*(0,0)}.\]
Then the same is true for the measure $2^{-\b}d\nu_1(x)=x^\b dx$ with support
$[0,1]$ (multiplying the measure by a constant does not change
the extremal polynomial for the Christoffel functions).
Next, consider the measure $d\nu_2(x)=|x|^\a dx$ with support $[-1,1]$. For this the
extremal polynomial for $\l_{2n}(\nu_2,0)$ is obtained from the
extremal polynomial for $\l_{n}(\nu_1,0)$  with $\b=(\a-1)/2$ by the substitution
$t\to t^2$ (see Section \ref{sectmodel1}, in particular see
the last paragraph in that section), i.e.
\[\frac{K_{2n}^{(2)}\left(t,0\right)}{K_{2n}^{(2)}(0,0)}=
\frac{K_{n}^{(1)}\left(t^2,0\right)}{K_n^{(1)}(0,0)}.\]
Hence, for even integers $n$
\[\frac{K_{n}^{(2)}\left(t,0\right)}{K_{n}^{(2)}(0,0)}=(1+o(1)) {\cal J}_{\frac{\a+1}2}(nt),\qquad |t|\le \frac{A}{n},\]
where
\be {\cal J}_{\frac{\a+1}2}(z):=\frac{\mathbb{J}_{\frac{\a-1}2}^*(z^2,0)}
{\mathbb{J}_{\frac{\a-1}2}^*(0,0)}=\frac{\mathbb{J}_{\frac{\a+1}2}^*(z)}
{\mathbb{J}_{\frac{\a+1}2}^*(0)}.\label{jdef}\ee

Fix a positive number $A$. According to what we have
just seen, for every even $n$
\bean \int_{-A/n}^{A/n}{\cal J}_{\frac{\a+1}{2}}(n t)^2|t|^\a dt&\le&
(1+o(1))\int_{-A/n}^{A/n}
\left(\frac{K_{n}^{(2)}\left(t,0\right)}{K_n^{(2)}(0,0)}\right)^2|t|^\a dt\\
&\le& (1+o(1))\l_n(\nu_2,0),\eean
and so for any (even) $n$
\[\int_{-A}^{A}{\cal J}_{\frac{\a+1}{2}}(x)^2|x|^\a dx= n^{\a+1}\int_{-A/n}^{A/n}{\cal J}_{\frac{\a+1}{2}}(n t)^2|t|^\a dt
  \le (1+o(1)) n^{\alpha + 1} \lambda_n\left(\nu_2,0\right).\]
Now if we let here $n\to\i$ and use the limit (\ref{nincs}) for the right-hand side,
then we obtain
\[ \int_{-A}^{A}{\cal J}_{\frac{\a+1}{2}}(x)^2|x|^\a dx\le L_\a,\]
where $L_\a$ is from (\ref{lalpha}).
Finally, since here $A$ is arbitrary, we can conclude
\be \int_{-\i}^{\i}{\cal J}_{\frac{\a+1}{2}}(x)^2|x|^\a dx\le L_\a.
\label{jfinite}\ee

\subsection{The upper estimate in Theorem \ref{thmain} for one arc}
The aim of this section is to construct polynomials
that verify the upper estimate for the Christoffel functions in Theorem
\ref{thmain} (which is the same as in Proposition \ref{thmainprop*})
when $\G$ consists of a single $C^2$-smooth arc, and $z_0\in\G$ is not
an endpoint of that arc.
In the next subsection  we shall indicate what to do when $\G$ has other components, as
well.

Let $\nu_\G$ be the equilibrium measure of $\G$ and $s_\G$ the arc measure
on $\G$. Since $\G$ is assumed
to be $C^{2}$-smooth, we have $d\nu_\G(t)=\o_\G(t)ds_\G(t)$
with an $\o_\G$ that is continuous and positive away
from the endpoints of $\G$  (see \cite[Proposition 2.2]{Totikarc}).

We may assume $z_0=0$ and that the real line is the tangent line
to $\G$ at the origin. By assumption, the measure $\mu$ we are
dealing with, is, in a neighborhood of the origin,
of the form $d\mu(z)=w(z)|z|^\a ds_\G(z)$ with some positive
and continuous function $w(z)$.

  Since $\G$ is assumed to be $C^2$-smooth, in
a neighborhood of the origin we have the parametrization
$\g(t)=\g_1(t)+i\g_2(t)$, $\g_1(t)\equiv t$, where $\g_2$ is a
twice continuously differentiable function such that $\g_2(0)=\g_2'(0)=0$.
In particular, as $t \to 0$ we have $\g_2(t)=O(t^2)$, $\g_2'(t)=O(|t|)$.
We shall also take an orientation of $\G$, and we shall
denote $z\prec w$ if $z\in \G$ precedes $w\in \G$ in that orientation.
We may assume that this orientation is such that around the origin
we have $z\prec w\Leftrightarrow \Re z<\Re w$.

It is known that, when dealing with $|z|^\a$ weights
on the real line, Bessel functions of the first kind
enter the picture, see \cite{Kuijlaars}, \cite{Lubinsky1}, \cite{Lubinsky2}.
For a given large $n$
we shall construct the necessary polynomials from two sources:
from points on $\G$ that follow the pattern of the zeros of the Bessel
function ${\cal J}_\frac{\a+1}{2}$, and from points that are obtained
from discretizing the equilibrium measure $\nu_\G$. The first type will
be used close to the origin (of distance $\le 1/n^\tau$ with some appropriate
$\tau$), while the latter type will be on the rest of  $\G$. So first
we shall discuss two different divisions of $\G$.

\subsubsection{Division based on the zeros of Bessel functions}
Let $\b=\frac{\a+1}2$ --- it is a positive number because $\a>-1$.
It is known that $J_\b$, and hence also ${\cal J}_\b$ from (\ref{jdef}),
has infinitely many positive zeros which are all simple and tend to infinity,
let them be $j_{\b,1}<j_{\b,2}<\ldots$. We have the asymptotic formula
(see \cite[15.53]{Watson})
\be j_{\b,k}=(k+\frac\b 2 -\frac14)\pi +o(1),\qquad k\to\i.\label{zeroasymp}\ee
The negative zeros of ${\cal J}_\b$
are $-j_{\b,k}$, and we have the product formula
(see \cite[15.41,(3)]{Watson})
\[J_\b(z)=\frac{(z/2)^\b}{\G(\b+1)}\prod_{k=1}^\i \left(1-\frac{z^2}{j_{\b,k}^2}\right).\]
Therefore,
\be {\cal J}_\b(z)=\prod_{k=1}^\i \left(1-\frac{z^2}{j_{\b,k}^2}\right).\label{ppr}\ee

Let $a_0=0$, and for $k>0$ let $a_k\in \G$ be the unique point on $\G$
such that $0\prec a_k$, and
\be \nu_\G(\ov{0a_k})=\frac{j_{\b,k}}{\pi n},\label{akdef}\ee
where $\ov{0a_k}$ denotes the arc of $\G$ that lies in between
$0$ and $a_k$. For negative $k$ let similarly $a_k$
be the unique number for which $a_k\prec 0$ and
\be \nu_\G(\ov{a_k0})=\frac{j_{\b,|k|}}{\pi n}.\label{akdef1}\ee
The reader should be aware that these $a_k$ and the whole division
depends on $n$, so a more precise notation would be $a_{k,n}$
for $a_k$, but we shall suppress the additional parameter $n$.

This definition makes sense only for finitely many $k$, say
for $-k_0<k<k_1$, and in view of (\ref{zeroasymp}) we have
$k_0+k_1=n+O(1)$, i.e. there are about $n$ such $a_k$ on $\G$.
The arcs $\ov{a_ka_{k+1}}$ are subarcs of $\G$
that
follow each other according to $\prec$, for them
\[\nu_\G(\ov{a_{k-1}a_{k}})=\frac{j_{\b,k}-j_{\b,k-1}}{\pi n},\qquad k>0,\]
\[\nu_\G(\ov{a_{k-1}a_{k}})=\frac{j_{\b,k+1}-j_{\b,k}}{\pi n},\qquad k<0,\]
and their union is almost
the entire $\G$:  there can be two additional arcs around the two endpoints
 with equilibrium measure $<(j_{\b,k_0}-j_{\b,k_0-1})/\pi n$
resp.
 $<(j_{\b,k_1}-j_{\b,k_1-1})/\pi n$.

\subsubsection{Division based solely on the equilibrium measure}
In this subdivision of $\G$ we follow the procedure in
\cite[Section 2]{Totikarc}.
Let $\ov{b_0b_1}\subset \G$ be the unique arc (at least for large $n$ it is unique)
with the property that $0\in \ov{b_0b_1}$,
$\nu_\G(\ov{b_0b_1})=1/n$, and if $\xi_0$ is the center of mass
of $\nu_\G$ on $\ov{b_0b_1}$, then $\Re \xi_0=0$.
For $k> 1$ let $b_k\in\G$ be the point on $\G$ (if there is one)
with the property that $0\prec b_k$ and
$\nu_\G(\ov{b_1b_k})=(k-1)/n$, and similarly, for
negative $k$ let $b_k\prec 0$ be the point on $\G$
with the property $\nu_\G(\ov{b_kb_0})=|k|/n$.
This definition makes sense only for finitely many $k$,
say for $-l_0<k<l_1$. Thus, the arcs $\ov{b_kb_{k+1}}$,
$-l_0<k<l_1-1$, continuously fill $\G_0$  (in the orientation of $\G_0$)
 and they all have equal, $1/n$ weight
 with respect to the equilibrium measure $\nu_\G$.
  It may happen
 that, with this selection, around the endpoints of $\G$ there still
 remain two ``little" arcs, say $\ov{b_{-l_0}b_{-l_0+1}}$
 and $\ov{b_{l_1-1}b_{l_1}}$ of $\nu_\G$-measure
  $<1/n$.
 We include also these two small arcs into our subdivision
 of $\G$, so in this case we divide $\G$ into $n+1$ arcs
 $\ov{b_kb_{k+1}}$, $k=-l_0,\ldots,l_1-1$.

Let $\xi_k$ be the center of mass of the measure $\nu_\G$ on
the arc
 $\ov{b_kb_{k+1}}$:
 \be \xi_k=\frac{1}{\nu_\G(\ov{b_kb_{k+1}})}\int_{\ov{b_kb_{k+1}}}u\;d\nu_\G(u).\label{xidef}\ee
Since the length of  $\ov{b_kb_{k+1}}$ is at most $C/n$ (note that $\o_\G$
has a positive lower bound), and $\G$ is $C^2$-smooth,
it follows that $\xi_k$ lies close
to the arc $\ov{b_kb_{k+1}}$:
\be {\rm dist}(\xi_k,\ov{b_kb_{k+1}})\le \frac{C}{n^2}\label{dist}\ee

For the polynomials
\be B_n(z)=\prod_{k\not=0}(z-\xi_k)\label{bnn}\ee
it was proven in \cite[Propositions 2.4, 2.5]{Totikarc} (see also
\cite[Section 2.2]{Totikarc}) that $B_n(z)/B_n(0)$ are uniformly bounded on $\G$:
\be \left |\frac{B_n(z)}{B_n(0)}\right|\le C_0,\qquad z\in \G.\label{bnbound}\ee

\subsubsection{Construction of the polynomials
${\cal C}_n$}
Choose a $0<\tau<1$ close to $1$ (we shall see later how close it has
to be to 1), and for an $n$ define $N=N_n=[n^{3(1-\tau)}]$.
We set
\be {\cal C}_n(z)=:\prod_{k=-N_n,\ k\not=0}^{N_n} \left(1-\frac{z}{a_k}\right)
\prod_{|k|>N_n} \left(1-\frac{z}{\xi_k}\right).\label{cndef}\ee
Note that the precise range of $k$ in the second factor
is $-l_0\le k<-N_n$ as well as $N_n<k\le l_1-1$.
Since the number of all $\xi_k$ is $n+1$, this polynomial has degree $n$,
and it takes the value 1 at the origin. This will be the main factor
in the test polynomial that will give the appropriate
upper bound for $\l_n(\mu,0)$, the other factor will
be the fast decreasing polynomial from Corollary \ref{fdp_beta}.

We estimate on $\G$ the two factors
\[{\cal A}_n(z):=\prod_{k=-N_n,\ k\not=0}^{N_n} \left(1-\frac{z}{a_k}\right)\]
and
\[{\cal B}_n(z):= \prod_{|k|>N_n} \left(1-\frac{z}{\xi_k}\right)\]
separately.
The estimates will be distinctly different for $|z|\le n^{-\tau}$ and
for $|z|>n^{-\tau}$.

\subsubsection{Bounds for ${\cal A}_n(z)$ for $|z|\le n^{-\tau}$}
In what follows, we shall use $N$ instead of $N_n$ ($=[n^{3(1-\tau)}]$).

Consider first
\[{\cal A}_n^*(x):=\prod_{k=1}^{N} \left(1-\frac{(n\pi\o_\G(0)x)^2}{j_{\b,k}^2}\right)\]
(recall that $j_{\b,k}$ are the zeros of the Bessel function $J_\b$ with $\b=(\a+1)/2$).
In view of (\ref{ppr}) we can write for real $|x|\le n^{-\tau}$
\[\frac{{\cal J}_\b(n\pi\o_\G(0) x)}{{\cal A}_n^*(x)}=\prod_{k>N} \left(1-\frac{(n\pi\o_\G(0) x)^2}{j_{\b,k}^2}\right).\]
Taking into account  (\ref{zeroasymp}), here
\[\frac{n\pi\o_\G(0) x}{j_{\b,k}}=O\left(\frac{n n^{-\tau}}{k}\right),\]
hence the product on the right is
\bean \exp\left\{ O\left(\sum_{k>N}\left(\frac{n n^{-\tau}}{k}\right)^2\right)\right\}
&=&\exp\left(O\left(\frac{n^{2(1-\tau)}}{N}\right)\right)\\
&=&\exp\left(O\left(\frac{1}{n^{(1-\tau)}}\right)\right)=1+o(1).\eean
Thus, our first estimate is
\be {\cal A}_n^*(x)=(1+o(1)){\cal J}_\b(n\pi\o_\G(0) x),\qquad |x|\le n^{-\tau}.\label{e1}\ee

Next, we go to a $z\in \G$ with $|z|\le n^{-\tau}$. Let $x$ be the real part of
$z$. Then, for $|z|\le n^{-\tau}$, we have (recall that $\G$ is $C^2$-smooth
and the real line is tangent to $\G$)
\[z=x+O(x^2)=x+O(n^{-2\tau}).\]

We shall need that the $a_k$'s with $|k|\le N$ are close to $j_{\b,k}/n\pi\o_\G(0)$.
To prove that, consider the parametrization $\g(t)=t+i\g_2(t)$ of $\G$ discussed in the beginning of this
section. Then $a_k=\g(\Re a_k)=\Re a_k+O((\Re a_k)^2)$. By the definition of the points
$a_k$ we have for $1\le k\le N$
\be \frac{j_{\b,k}}{\pi n}=\nu_\G(\ov {0a_k})=\int_{0}^{\Re a_k}
\o_\G(\g(t))|\g'(t)|dt.\label{g1}\ee
Now we use that around the origin $\o_\G$ is $C^1$-smooth (see \cite[Proposition 2.2]{Totikarc}),
hence on the right
\[\o_\G(\g(t))=\o_\G(0)+O(|\g(t)|)=\o_\G(0)+O(|t|),\]
while
\[|\g'(t)|=\sqrt{1+\g_2'(t)^2}=\sqrt{1+O(t^2)}=1+O(t^2),\]
 hence
\[\frac{j_{\b,k}}{\pi n}=\int_{0}^{\Re a_k}
\left(\o_\G(0)+O(|t|)\right)dt=\o_\G(0)\Re a_k+O((\Re a_k))^2, \]
which implies
\be \Re a_k=\frac{j_{\b,k}}{\pi n\o_\G(0)}+
O\left(\left(j_{\b,k}/n\right)^2\right). \label{g2}\ee

Therefore, since here $j_{\b,k}\le Ck$ (see  (\ref{zeroasymp})),
\be a_k-\frac{j_{\b,k}}{n\pi\o_\G(0)}=
(a_k-\Re a_k)+\Re a_k-\frac{j_{\b,k}}{n\pi\o_\G(0)}
=O\left(\left(k/n\right)^2\right).\label{ffg}\ee

Let
\be \r=(\a+9)(1-\tau)\label{rhodef},\ee
and suppose that
\be \left|x-\frac{j_{\b,k}}{n\pi\o_\G(0)}\right|\ge \frac{1}{n^{1+\r}}, \qquad
\mbox{for all $-N\le k\le N$}.\label{zzn}\ee
Then in the product
\[\frac{{\cal A}_n(z)}{{\cal A}_n^*(x)}=
\prod_{k=-N,\ k\not=0}^N \frac{1-z/a_k}{1-n\pi\o_\G(0) x/j_{\b,k}}
=\prod_{k=-N,\ k\not=0}^N \frac{j_{\b,k}-zj_{\b,k}/a_k}{j_{\b,k}-n\pi\o_\G(0) x}
\]
all denominators are $\ge c/n^\r$. As for the numerators,
we have (recall (\ref{ffg}) and $|a_k|\ge ck/n$)
\[|j_{\b,k}/a_k-n\pi\o_\G(0)|= O(k),\]
 and hence, because of $z=x+O(x^2)$,
\bean |zj_{\b,k}/a_k-n\pi\o_\G(0)x|&=& O(|z|k+n x^2)=O(Nn^{-\tau}+nn^{-2\tau})\\
&=&O(n^{3-4\tau}+n^{1-2\tau})=O(n^{3-4\tau}).\eean
Therefore, for the individual factors in ${\cal A}_n(z)/{\cal A}^*_n(z)$ we have
\[\frac{j_{\b,k}-zj_{\b,k}/a_k}{j_{\b,k}-n\pi\o_\G(0) x}
=1+O(n^{3-4\tau}n^\r),\]
from which  we can conclude
\bean \frac{{\cal A}_n(z)}{{\cal A}_n^*(x)}&=& \left(1+O(n^{3-4\tau}n^\r)\right)^{2N}
=\exp\left(O(n^{3-4\tau}n^\r N)\right)\\
&=&\exp\left(O(n^{6-7\tau+\r}\right))
=\exp\left(O(n^{(15+\a)(1-\tau)-\tau})\right)=1+o(1)\eean
provided
\be (15+\a)(1-\tau)<\tau.\label{prov1}\ee

Let $\G_n$ be the set of those $z\in \G$ for which $|z|\le n^{-\tau}$ and
(\ref{zzn}) is true with $x=\Re z$:
\be \G_n=\{z\in \G\sep |z|\le n^{-\tau},\ \ \mbox{(\ref{zzn}) is true with $x=\Re z$}\}.
\label{Gamman}\ee
So far we have proved
(see (\ref{e1}) and the preceding estimates)
\be {\cal A}_n(z)=(1+o(1)){\cal J}_\b(n\pi\o_\G(0) x),\qquad z\in \G_n.\label{e2}\ee
$\G_n$ is a subset of the arc $\G\cap \D_{n^{-\tau}}(0)$
of $s_\G$-measure at most $O(Nn^{-1-\r})=O(n^{2-3\tau-\r})$,
so its relative measure compared to the $s_\G$-measure of
$\G\cap \D_{n^{-\tau}}(0)$ is at most
\[O(n^{2-3\tau-\r+\tau})=O(n^{2-2\tau-\r})=o(N^{-2})\]
because
\[2-2\tau-\r=-(\a+7)(1-\tau)<-6(1-\tau).\]
Since ${\cal A}_n$ has degree $2N$,
from the  Remez-type inequality in Lemma \ref{Remez} we can conclude
that
\[\sup\{|{\cal A}_n(z)|\sep z\in \G\cap \D_{n^{-\tau}}(0)\}
\le (1+o(1))\sup\{|{\cal A}_n(z)|\sep z\in \G_n\}.\]
But ${\cal J}_\b(t)$ is bounded on the whole real line
(see \cite[Section 7.21]{Watson}), therefore we get from
here and from  (\ref{e2})
that there is a constant $C_1$ such that
\be |{\cal A}_n(z)|\le C_1\label{e3}\ee
for all $z\in \G$, $|z|\le n^{-\tau}$.

\subsubsection{Bounds for ${\cal B}_n(z)$ for $|z|\le n^{-\tau}$}
Consider now, for $z\in \G$, $|z|\le n^{-\tau}$,
the expression
\[{\cal B}_n(z)= \prod_{|k|>N} \frac{\xi_k-z}{\xi_k}.\]
Recall that the smallest and largest indices here (they are
$k_{-l_0}$ and $k_{l_1}$) refer to a $\xi_k$ that were
selected for the two additional intervals around the endpoints
of $\G$, hence for them we have
\[\frac{\xi_k-z}{\xi_k}=1+O(|z|)=1+o(1),\qquad k=-l_0,l_1-1.\]
The rest of the indices refer to points $\xi_k$ which were the center
of mass on the arcs $\ov{b_kb_{k+1}}$ which have $\nu_\G$-measure
equal to $1/n$. We are going to compare
$\log|z-\xi_k|$ with the average of $\log|z-t|$ over the
arc $\ov{b_kb_{k+1}}$ with respect to $\nu_\G$:
\[\log|z-\xi_k|-n\int_{\ov{b_kb_{k+1}}}\log|z-t|d\nu_\G(t)
=-n\int_{\ov{b_kb_{k+1}}}\log\left|\frac{z-t}{z-\xi_k}\right|d\nu_\G(t).\]
Here
\[\frac{z-t}{z-\xi_k}=1+\frac{\xi_k-t}{z-\xi_k},\]
and for $t\in \ov{b_kb_{k+1}}$, in the numerator
$|\xi_k-t|\le C/n$. Since $|z|$ is small (at most $n^{-\tau}$) and
compared to that $|\xi_k|$ is large ($\ge N/n=n^{2(1-\tau)-\tau}$),
the second term on the right is small in absolute value,
hence
\[\log\left|\frac{z-t}{z-\xi_k}\right|=\Re \log\left(1+\frac{\xi_k-t}{z-\xi_k}\right)
=\Re\frac{\xi_k-t}{z-\xi_k}+
O\left(\left|\frac{\xi_k-t}{z-\xi_k}\right|^2\right).\]
Therefore,
\bean  \left|n\int_{\ov{b_kb_{k+1}}}\log\left|\frac{z-t}{z-\xi_k}\right|d\nu_\G(t)\right|
&=&n\int_{\ov{b_kb_{k+1}}}
O\left(\left|\frac{\xi_k-t}{z-\xi_k}\right|^2\right)d\nu_\G(t)\\
&=&O\left(\frac{(1/n)^2}{(k/n)^2}\right)=O\left(\frac{1}{k^2}\right),\eean
because the integral
\[\int_{\ov{b_kb_{k+1}}}\Re \frac{\xi_k-t}{z-\xi_k}d\nu_\G(t)=\Re\frac{1}{z-\xi_k}\int_{\ov{b_kb_{k+1}}} (\xi_k-t)d\nu_\G(t)\]
vanishes by the choice of $\xi_k$.

Hence, if
\[H_n=\bigcup_{-l_0<k<-N,\ N<k<l_1-2}\ov{b_kb_{k+1}},\]
then
\bean \log \prod_{|k|>N} |\xi_k-z|-n\int_{H_n}\log|z-t|d\nu_\G(t)&=&o(1)+O\left(
\sum_{|k|>N}k^{-2}\right)\\
&=&o(1)+O(N^{-1})=o(1).\eean
If we set here $z=0$, then we get
\[\log \prod_{|k|>N} |\xi_k|-n\int_{H_n}\log|t|d\nu_\G(t)=o(1).\]
Therefore,
\be \log |{\cal B}_n(z)|-n\int_{H_n}\log\frac{|z-t|}{|t|}d\nu_\G(t)=o(1).
\label{aah}\ee

As the whole integral
\[\int_{\G}\log\frac{|z-t|}{|t|}d\nu_\G(t)\]
is the value of the logarithmic potential of the equilibrium
measures $\nu_\G$ in two points of
$\G$, and since this logarithmic potential is constant on $\G$
by Frostman's theorem (\cite[Theorem 3.3.4]{Ransford}), we obtain
that this whole integral is 0, and so (\ref{aah}) is
equivalent to
\be \log |{\cal B}_n(z)|+n\int_{\G\setm H_n}\log\frac{|z-t|}{|t|}d\nu_\G(t)=o(1).
\label{aah1}\ee
The set $\G\setm H_n$ consists of the two small additional arcs $\ov{b_{-l_0}b_{-l_0+1}}$,
$\ov{b_{l_1-1}b_{l_1}}$ and of the "big" arc $\ov{b_{-N}b_{N+1}}$. The integral,
more precisely, $n$-times the integral, on
the left over the two small arcs is $o(1)$ (recall that $|z|$ is small,
while on those arcs $|t|$ stays away from 0), and now we estimate the integral over
the "big" arc, i.e. we consider
\be n\int_{\ov{b_{-N}b_{N+1}}}\log\frac{|z-t|}{|t|}d\nu_\G(t)
=n\int_{\Re b_{-N}}^{\Re b_{N+1}}\log\frac{|z-\g(t)|}{|\g(t)|}\o_\G(t)|\g'(t)|dt.
\label{kjh}\ee

By the definition of the points $b_k$ we have $b_1=(1/2+o(1))/n$,
\[\frac{N}{n}=\nu_\G(\ov{b_1b_{N+1}})=\int_{\Re b_1}^{\Re b_{N+1}}
\o_\G(\g(t))|\g'(t)|dt\]
and the same reasoning as in between (\ref{g1}) and (\ref{g2}) yields from this
that
\[\Re b_{N+1}=\frac{N+\frac12}{n\o_\G(0)}+O\left((N/n)^2\right).\]
We get similarly
\[\Re b_{-N}=\frac{-N+\frac12}{n\o_\G(0)}+O\left((N/n)^2\right).\]

If $z=\g(\zeta)=\z+i\g_2(\z)$, then in the integrand in (\ref{kjh}) we
have
\[\o_\G(\g(t))=\o_\G(0)+O(|t|),\quad |\g'(t)|=1+O(t^2),\]
\[\log|\g(t)|=\log(|t|+O(t^2))=\log|t|+O(|t|),\]
and (with $\g(t)=t+i\g_2(t)$)
\[\log|\g(\zeta)-\g(t)|=\log\sqrt{(\zeta-t)^2+(\g_2(\zeta)-\g_2(t))^2},\]
where
\[\g_2(\zeta)-\g_2(t) =\g_2'(\zeta)(\zeta-t)+O((\zeta-t)^2)=O(|\zeta||\zeta-t|)
+O((\zeta-t)^2).\]
Therefore, since $|\zeta|\le n^{-\tau}$ and  $|\z-t|\le CN/n$, we have
\[\log|\g(\zeta)-\g(t)|=\log|\z-t|+O(n^{-2\tau}) +O\left((N/n)^2\right).\]

By substituting all these into (\ref{kjh}) we obtain that
with
\[M_1=(-N+1/2)/n\o_\G(0),\qquad M_2=(N+1/2)/n\o_\G(0),\]
 the expression
in $(\ref{kjh})$ is equal to
\[n\int_{M_1+O((N/n)^2)}^{M_2+O((N/n)^2))}
\log\left|\frac{\zeta -t}{t}\right|\o_\G(0)dt=
n\int_{M_1}^{M_2}
\log\left|\frac{\zeta -t}{t}\right|\o_\G(0)dt+O((N/n)^2)\]
plus an error term which is at most
\bean nO\left((N/n)^2\right)
&+&nO\left(N/n\right)O(n^{-2\tau})+
nO\left((N/n)^3\right)\\
&=&O(n^{6(1-\tau)-1})+O(n^{3(1-\tau)-2\tau})+O(n^{9(1-\tau)-2})=o(1)\eean
if (\ref{prov1}) is satisfied.

From what we have done so far, it follows, say, for $0\le \zeta=\Re z \le n^{-\tau}$,
that with $M=N/n\o_\G(0)$
\[  \log |{\cal B}_n(z)|=o(1)-n\o_\G(0)\int_{-M}^{M}
(\log|\zeta -t|-\log |t|)dt.\]
But
\bean \int_{-M}^{M}
(\log|\zeta -t|-\log |t|)dt =\int_{M-\zeta}^{M}\log\frac{u+\zeta}{u}du
&=&\int_{M-\zeta}^{M}O\left(\frac\zeta{u}\right)du\\
&=&O(\z^2/M)=O(\z^2n/N),\eean
hence
\bean \log |{\cal B}_n(z)|=
O(n\zeta^2 (n/N))+o(1)&=&O(n^{2-2\tau -3(1-\tau)})+o(1)\\
&=&O(n^{-(1-\tau)})+o(1)=o(1)\eean
for all $z\in \G$, $|z|\le n^{-\tau}$, provided $\tau$ satisfies
(\ref{prov1}). Thus, in this case (i.e. when $|z|\le n^{-\tau}$)
\be |{\cal B}_n(z)|=1+o(1).\label{e4}\ee

All the reasonings so far used the assumption (\ref{prov1}), which
can be satisfied by choosing $\tau<1$ sufficiently close to 1.

\subsubsection{The square integral of ${\cal C}_n$ for $|z|\le n^{-\tau}$}
Using (\ref{e2}), (\ref{e3}) and (\ref{e4})
we can now estimate the square integral of $|{\cal C}_n(z)|$ against the measure
$\mu$ over the arc $\G\cap \D_{n^{-\tau}}(0)$. Indeed, let $\Re \G_n$
be the projection of $\G_n$ (see (\ref{Gamman})) onto the real line. Then $\Re \G_n$
is an interval $[-\a_n,\b_n]$ minus all the intervals
\[I_k=\left(\frac{j_{\b,k}}{n\pi\o_\G(0)}-\frac{1}{n^{1+\r}},
\frac{j_{\b,k}}{n\pi\o_\G(0)}+\frac{1}{n^{1+\r}}\right).\]
Here $\a_n,\b_n\sim n^{-\tau}$, and the $|k|$ in these latter
intervals is at most $2n^{1-\tau}$ (see (\ref{zeroasymp})). Therefore (use also that
\[d\mu(z)=w(z)|z|^\a ds_\G(z)=(1+o(1))w(0)|x|^\a dx\]
 and that $|\g'(t)|=1+o(1)$ for $t=O(n^{-\tau})$),
\bean \int_{\G\cap \D_{n^{-\tau}}(0)}|{\cal C}_n(z)|^2d\mu(z)
&=&(1+o(1))\int_{\Re \G_n} {\cal J}_\b(n\pi\o_\G(0) x)^2 w(0) |x|^\a dx\\
&+&C\int_{ \cup_k I_k} |x|^\a dx.\eean
In view of (\ref{jfinite}) the first integral is at most
\[\frac{(1+o(1))w(0)}{(n\pi \o_\G(0))^{\a+1}}L_\a\]
with the $L_\a$ defined in (\ref{lalpha}).
The second integral is at most
\[C\sum_{k=1}^{2n^{1-\tau}}\frac{1}{n^{1+\r}} \left(\frac{k}{n}\right)^{\a}
=O(n^{(1-\tau)(\a+1)-\a-1-\r})=o(n^{-\a-1})\]
because of (\ref{rhodef}).

Combining these we can see that
\be \limsup_{n\to\i}n^{\a+1}\int_{\G\cap \D_{n^{-\tau}}(0)}|{\cal C}_n(z)|^2d\mu(z)
\le \frac{w(0)L_\a}{(\pi \o_\G(0))^{\a+1}} .\label{veg1}\ee

\subsubsection{The estimate of ${\cal C}_n(z)$ for $|z|>n^{-\tau}$}
Now let $z\in \G$, $|z|>n^{-\tau}$, say $0\prec z$.
In view of (\ref{zeroasymp}) and of the definition of the points $a_k$ and $b_k$,
\[\nu_\G(\ov{0a_k})=\frac{k}{n}+O(n^{-1}),\qquad
\nu_\G(\ov{0b_k})=\frac{k}{n}+O(n^{-1}), \quad k>0.\]
A similar relation holds for negative $k$. These imply
\be a_k-b_k=O(n^{-1}),\label{iio}\ee
and so there is an
integer $T_0$ (independent of $n$) such that
\[b_{k-T_0}\prec a_k \prec b_{k+T_0}\qquad{\mbox{for $k>T_0$}}\]
and similarly
\[b_{-k-T_0}\prec a_{-k} \prec b_{-k+T_0}\qquad{\mbox{for $k>T_0$}}.\]
Since $\G$ is $C^2$-smooth, this implies
the existence of a $\d>0$ and a $T$ (actually, $T=T_0+1$ will suffice)
such that if $|z|\le \d$ (and $z$ satisfying also the previous
condition that $z\in \G$, $0\prec z$)
\begin{description}
\item[(i)]  then $z\preceq a_k$, $T<k\le N$ imply
\[ |z-a_k|<|z-\xi_{k+T}|,\qquad |a_k|>|\xi_{k-T}|\]
\item[(ii)] then $a_k\prec z$,  $T<k\le N$ imply
\[ |z-a_k|<|z-\xi_{k-T}|,\qquad |a_k|>|\xi_{k-T}|,\]
\item[(iii)] then $a_k\prec z$,  $-N\le k<-T$ imply
\[ |z-a_k|<|z-\xi_{k-T}|,\qquad |a_k|>|\xi_{k+T}|.\]
\end{description}

For this particular $z\in \G$, $0\prec z$, $\d>|z|> n^{-\tau}$
we shall compare the value $|{\cal C}_n(z)|$ with the value of a
modified polynomial $|\tilde {\cal C}_n(z)|$, which we obtain
as follows. Remove all factors $|1-z/a_k|$ from $|{\cal C}_n(z)|$
with $|k|\le T$, then
\begin{description}
\item[(i')] for $z\preceq a_k$,  $T<k\le N$   replace the factor
$|1-z/a_k|=|a_k-z|/|a_k|$ in $|{\cal C}_n(z)|$ by $|z-\xi_{k+T}|/|\xi_{k-T}|$
\item[(ii')] for $a_k\prec z$,  $T<k\le N$ replace the factor
$|a_k-z|/|a_k|$  in $|{\cal C}_n(z)|$ by $|z-\xi_{k-T}|/|\xi_{k-T}|$,
\item[(iii')] for  $a_k\prec z$,  $-N\le k<-T$  replace the factor
$|a_k-z|/|a_k|$  in $|{\cal C}_n(z)|$ by $|z-\xi_{k-T}|/|\xi_{k+T}|$.
\end{description}

Removing a factor $|1-z/a_k|$ from  $|{\cal C}_n(z)|$
decreases the absolute value of the polynomial by at most
a factor
$1/C_2n$ with some $C_2$ because each $a_k$, $k\not=0$ is $\ge c/n$ in absolute value.
On the other hand, the replacements in {\bf (i')}--{\bf (iii')}
increase the absolute value of the polynomial
at $z$ because of  {\bf (i)}--{\bf (iii)}. Hence,
\[|{\cal C}_n(z)|\le C_3n^{2T} |\tilde {\cal C}_n(z)|.\]
But $|\tilde {\cal C}_n(z)|$ has the form
\[|\tilde {\cal C}_n(z)|=\frac{\prod ^*|z-\xi_k|}{\prod ^{**}|\xi_k|},\]
where all $|z-\xi_k|$, $-l_0\le k<l_1$, appear in $\prod^*$
except at most $5T$ of them (at most $2T$ around $z$, at most $2T$ around 0,
and at most $T$ around $a_N$),  and where
some $|z-\xi_k|$ may appear
twice, but at most $T$ of them (all around $a_N$). Therefore,
if $z$ also satisfy $|z-\xi_k|\ge n^{-4}$ for all $-l_0\le k\le l_1-1$, then
\[\prod{}^*|z-\xi_k|\le \left(\prod_{k=-l_0, k\not=0}^{l_1-1} |z-\xi_k|\right)({\rm diam} \G)^{T}(n^4)^{5T}.\]

A similar reasoning gives that in $\prod^{**}$ all $|\xi_k|$
appear except perhaps $2T$ of them, and none of the
$\xi_k$ is repeated twice, therefore,
\[\prod{}^{**}|\xi_k|\ge \left(\prod_{k=-l_0, k\not=0}^{l_1-1} |\xi_k|\right)
\frac{1}{({\rm diam}(\G))^{2T}}.\]

Therefore,
\[|{\cal C}_n(z)|\le C_3n^{2T} |\tilde {\cal C}_n(z)|
\le C_4 n^{22T}\prod_{k=-l_0, k\not=0}^{l_1-1} \frac{|z-\xi_k|}{|\xi_k|}
.\]
But the product on the right is $|B_n(z)/B_n(0)|$
with $B_n$ from (\ref{bnn}), for which
the bound (\ref{bnbound}) is true.
Hence, we can conclude
\be |{\cal C}_n(z)|\le C_5n^{22T},\label{ccbound}\ee
under the condition that
$|z-\xi_k|\ge n^{-4}$ is true for all $k$.

This reasoning was made for $|z|\le \d$ and $0\prec z$.
The case $|z|\le \d$, $z\prec 0$ is completely similar.
On the other hand, if $z\in \G$,
$|z|> \d$, then we use
for all $-N\le k\le N$, $k\not=0$
\[ |z-a_k|=|z-\xi_{k}+O(n^{-1})|
=|z-\xi_{k}|(1+O(n^{-1}))|\]
because  all $a_k$, $\xi_k$ with
$|k|\le N$ lie of distance $\le CN/n=O(n^{3(1-\tau)-1})=o(1)$ from the origin.
Thus, if we replace every $|z-a_k|$ in ${\cal C}_n(z)$, $|k|\le N$, $k\not=0$
 by
$|z-\xi_k|$, then under this replacement, the value of
the polynomial can decrease by at most a factor
$(1+O(n^{-1}))^n=O(1)$. We also want to replace each $|a_k|$ by $|\xi_k|$:
\[\prod_{k=1}^N|a_k|\ge \prod_{k=1}^T|a_k|\prod_{k=T+1}^{N}|\xi_{k-T}|\ge
cn^{-T}\prod_{k=1}^{N}|\xi_k|\]
because  $|a_k|\ge |\xi_{k-T}|$ for $k>T$ and $|a_k|\ge c/n$ for all $k\not=0$.
A similar estimate holds for negative
values, by which we get
\[|{\cal C}_n(z)|\le C n^{2T}\prod_{k\not =0} \frac{|z-\xi_k|}{|\xi_k|}
\le CC_0n^{2T},\]
since the last product is just $| B_n(z)/B_n(0)|$
for which we can use (\ref{bnbound}).

Therefore, for such values (i.e. for $|z|>\d$)
we can again claim the bound (\ref{ccbound}).

All in all, we have proven
(\ref{ccbound}) on $\G$ with the exception of
those $z\in \G$ for which there is
a $\xi_k$ such that
$|z-\xi_k|<n^{-4}$. This exceptional set has
arc measure at most $Cn\cdot n^{-4}=Cn^{-3}$,
so an application of Lemma \ref{Remez} gives
that the bound
\be |{\cal C}_n(z)|\le C_5n^{22T},\label{ccbound1}\ee
holds throughout the whole
$\G$.

\subsubsection{Completion of the upper estimate for a single arc}
Let
\[P_n(z)={\cal C}_n(z)S_{n,0,\G}(z),\]
where ${\cal C}_n(z)$ is as in (\ref{cndef}) and
$S_{n,0,\G}(z)$ is the fast decreasing polynomial from Corollary
\ref{fdp_beta} for $K=\G$ and for the point 0.
This $P_n$ has degree $(1+o(1))n$, its value is 1 at the origin,
and $|P_n(z)|\le |{\cal C}_n(z)|$ on $\G$.
On $\G\cap \D_{n^{-\tau}}(0)$ we just use
$|P_n(z)|\le |{\cal C}_n(z)|$,
while for $|z|> n^{-\tau}$ we get from (\ref{ccbound1})
and (\ref{asd})
that
\[|P_n(z)|\le 2C_5n^{22T}C_\tau e^{- c_\tau n^{\tau_0}}=o(n^{-\a-1}).\]
As a consequence,
\[ \limsup_{n\to\i}n^{\a+1}\int_\G|P_n(z)|^2d\mu(z)
\le \limsup_{n\to\i}n^{\a+1}\int_{\G\cap \D_{n^{-\tau}}(0)}|{\cal C}_n(z)|^2d\mu(z).\]
Since the integral on the left is an upper bound for
$\l_{{\rm deg}(P_n)}(\mu,0)$,
we obtain from (\ref{veg1}) (use also (\ref{limsup}))
\be \limsup_{n\to\i}n^{\a+1}
\l_n(\mu,0)\le \frac{w(0)L_\a}{(\pi \o_\G(0))^{\a+1}} .\label{veg2}\ee
This proves one half of Proposition \ref{thmainprop*} for a single arc.\endproof

\subsection{The upper estimate for several components}
In this section, we sketch what to do with the
preceding reasoning when $\G$ may have several components
which can be $C^2$ Jordan curves or  arcs.
Let $\G_0,\ldots,\G_{k_0}$ be the different components
of $\G$, and assume that $z_0=0$ belongs to $\G_0$.
Assume, that this $\G_0$ is a Jordan
arc, actually this is the only case we shall use below
i.e. when $z_0$ belongs to an arc component of $\G$,
and the other components are Jordan curves.
On this $\G_0$ we introduce the points $a_k$ as
before, there is no need for them on the other components
of $\G$ (they played a role above only in a
small neighborhood of $0$).

On the other hand,
on the whole $\G$ we introduce the analogue of the points
$\xi_k$ by repeating the process in \cite[Section 2]{Totikarc}.
The outline is as follows.
Let $\theta_j=\nu_\G(\G_j)$,
consider the integers $n_j=[\theta_j n]$, and
divide each $\G_j$, $j>0$,
into $n_j$ arcs $I_k^j$ each having equal weight $\theta_j/n_j$
with respect to $\nu_\G$, i.e. $\nu_\G(I_k^j)=\theta_j/n_j$.
On $\G_0$
introduce the points $b_k$ as before, and the
arcs $I_k^0= \ov{b_kb_{k+1}}$.
 Let $\xi_k^j$ be the center of mass of the
arc $I_k^j$ with respect to $\nu_\G$, and consider the polynomial
\be R_n(z)=\prod_{j,k}(z-\xi_k^j)\label{rndef0}\ee
of degree at most $n+O(1)$.
 Now the polynomial
\be B_n(z)=R_n(z)/(z-\xi_0^0)\label{pdef0}\ee
will have similar properties as the $B_n$ before,
namely (\ref{bnbound}) is true,
see \cite[Section 2]{Totikarc}, in particular
see \cite[Propositions 2.4 and 2.5]{Totikarc}.

The rest of the argument in the preceding subsections
does not change: the components of $\G_l$, $l\ge 1$ are
far from $z_0=0$, the corresponding estimates
in the above proof on them is the same as
the estimate in the preceding subsections for $|z|> \d$.\endproof

\subsection{The lower estimate in Theorem \ref{thmain} on Jordan arcs}
In this section, the assumption is the same as before,
namely that $\G$ consists of finitely many
$C^2$-smooth Jordan arcs and curves, $z_0$ belongs
to an arc component of $\G$ and $\mu$ is given by
(\ref{general_mu_def}). Our aim is to prove
the necessary lower bound for $\l_n(\mu,z_0)$.

In this proof we shall closely follow
the proof of \cite[Theorem 3.1]{Totikarc}.

Let $\O$ be the unbounded component of $\ov \C\setm \G$,
and denote by $g_{\O}$ the Green's function
of $\O$ with respect to the pole at infinity (see
e.g. \cite[Sec. 4.4]{Ransford}).

Assume to the contrary, that there are infinitely many $n$ and for each
$n$ a polynomial $Q_n$ of degree at most $n$ such that $Q_n(z_0)=1$ and
\be n^{1+\a}\int |Q_n|^2d\mu<(1-\d)\frac{w(z_0)L_\a}{(\pi \o_\G(z_0))^{\a+1}}\label{q1}\ee
with some $\d>0$, where $L_\a$
was defined in (\ref{lalpha}). The strategy will be to show that this implies the following:
{\it there exists another system $\G^*$ of piecewise $C^2$-smooth
Jordan  curves  and an extension of $w$ to $\G^*$
such that $\G\subseteq \G^*$,
in a neighborhood $\D_0$ of $z_0$ we have $\G\cap \D_0=\G^*\cap \D_0$,
and for the  measure
\be d\mu^*(z) = w(z) |z-z_0|^{\alpha} ds_{\G^*}(z)\label{mu*}\ee
with support  $\G^*$
\be \liminf_{n\to\i}n^{1+\a}\l_n(\mu^*,z_0) <\frac{w(z_0)L_\a}{(\pi \o_{\G^*}(z_0))^{\a+1}}.
\label{th1**1}\ee}
Since this contradicts Proposition \ref{thmainprop}, (\ref{q1}) cannot
be true.

Let $\G_0,\ldots,\G_{k_0}$ be the connected components of $\G$, $\G_0$ being
the one that contains $z_0$. We shall only consider the case when
$\G_0$ is a Jordan arc, when $\G_0$ is a Jordan curve, the argument is
similar, see \cite[Section 3]{Totikarc}.

Let ${\bf n}_\pm$ be the two normals to $\G_0$ at $z_0$, and let
$A_\pm=\partial g_{\O}(z_0)/\partial {\bf n}_\pm$ be the corresponding
normal derivatives of the Green's function of $\O$
with pole at infinity. Assume, for example, that $A_+\ge A_-$. Note that
 $A_->0$, see \cite[Section 3]{Totikarc}.

Let $\e>0$ be an arbitrarily small number.
For each $\G_j$ that is a
 Jordan  arc, connect the two endpoints of $\G_j$
by another $C^2$-smooth Jordan arc $\G_j'$ that lies close to $\G_j$
so that we obtain a system $\G'$ of $k_0+1$ Jordan curves
with boundary $(\cup_j \G_j)\bigcup(\cup_j\G_j')$. Assume also that $\G_0'$ is
selected so that ${\bf n}_+$ is the outer normal to $\G'$ at $z_0$.
This can be done in such a way
that (with $\O'$ being the unbounded component of $\ov\C\setm \G'$)
\be \frac{\partial g_{\O'}(z_0)}{\partial {\bf n}_+}>
\frac{1}{1+\e}
\frac{\partial g_{\O}(z_0)}{\partial {\bf n}_+},\label{gk}\ee
see \cite[Section 3]{Totikarc}.

Select a small disk $\D_0$ about $z_0$ for which $\G'\cap \D_0=\G\cap \D_0$, and, as in
\cite[Section 3]{Totikarc}, choose  a lemniscate $\s=\{z\sep |T_N(z)|=1\}$
(with some polynomial $T_N$ of degree equal to some integer $N$)
such that $\G'$ lies in the interior of $\s$ (i.e.
in the union of the bounded components of $\C\setm \s$)
except for the point
$z_0$, where $\s$ and $\G'$ touch each other, and (with $\O_\s$ being the unbounded
component of $\ov\C\setm \s$)
\be \frac{\partial g_{\O_\s }(z_0)}{\partial {\bf n}_+}>\frac{1}{1+\e} \frac{\partial g_{\O}(z_0)}{\partial {\bf n}_+}.\label{gk0}\ee
For the Green's function associated with the
outer domain $\O_\s$ of $\s$
 we have (see \cite[(3.6)]{Totikarc})
\be \frac{\partial g_{\O_\s }(z_0)}{\partial {\bf n}_+}=\frac{|T_N'(z_0)|}{N}.\label{gn}\ee

\begin{figure}
\centering
\includegraphics[scale=0.5]{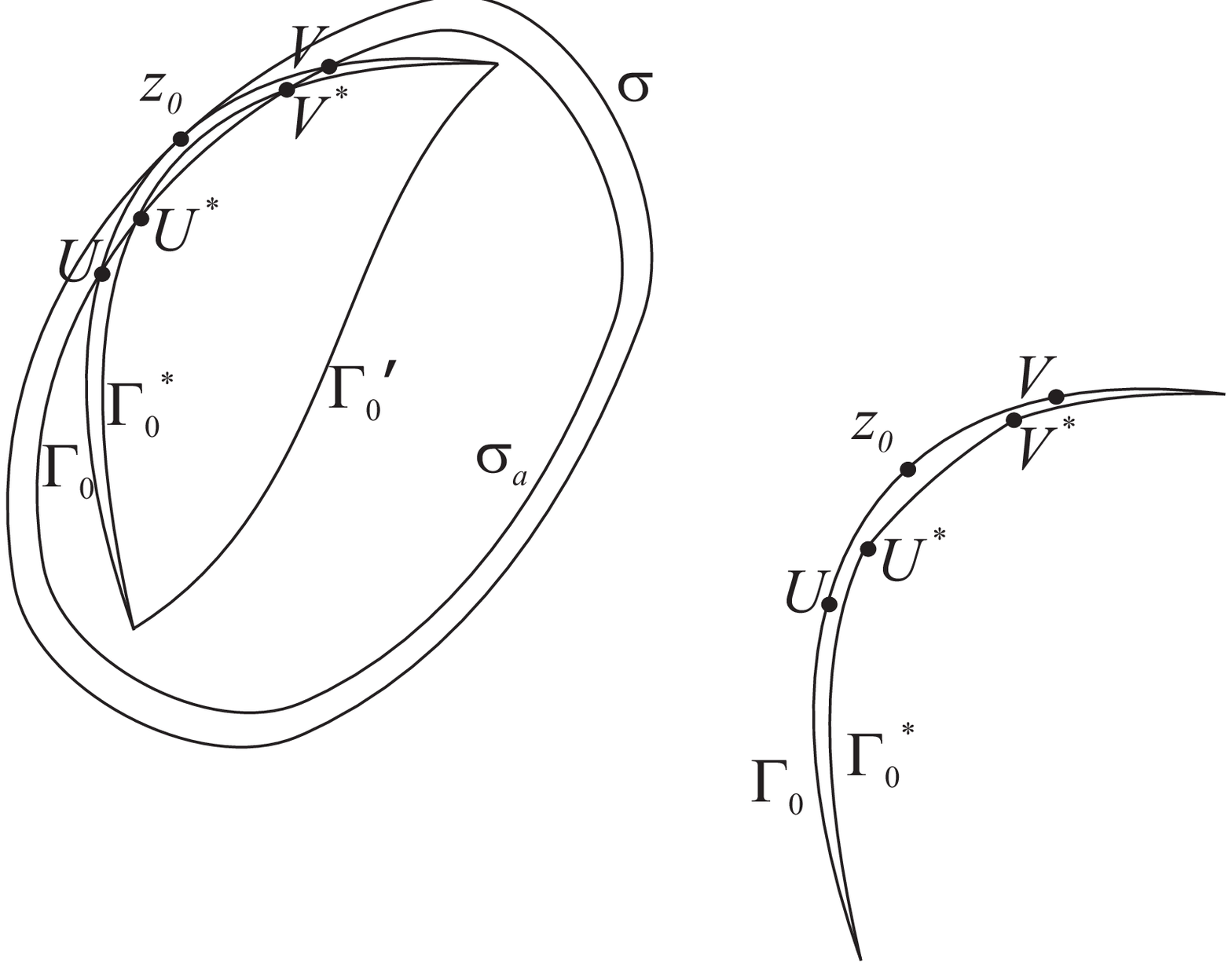}
\caption{\protect\label{fig0}}
\end{figure}
For a small $a$ let $\s_a$ be the lemniscate
$\s_a:=\{z\sep |T_N(z)|=e^{-a}\}$. According to \cite[Section 3]{Totikarc},
if $\D\subset \D_0$ is a fixed small neighborhood of $z_0$,
then for sufficiently small $a$ this $\s_a$ contains $\G'\setm \D$ in its interior, while in
$\D$ the two curves $\G_0$ and $\s_a$ intersect in two points $U,V$,
see Figure \ref{fig0}.
The points $U$ and $V$ are connected by the arc $\overline{UV}_{\G_0}$
on $\G_0$ and also by the arc
$\overline{UV}_{\s_a}$  on $\s_a$ (there are actually two such arcs on $\s_a$, we take
the one lying in $\D$). For each $\G_j$ which is a Jordan arc
connect the two endpoints
of $\G_j$ by a new $C^2$ Jordan arc $\G_j^*$ going inside $\G'$ so that
on $\G_j^*$ we have
\be g_{\O}(z)\le a^2,\qquad z\in \G_j^*.\label{gj*}\ee
In addition, $\G_0^*$ can be selected so that in $\D$ it intersects $\s_a$ in
two points $U^*,V^*$. Then $\overline{U^*V^*}_{\s_a}$ is a subarc of
$\overline{UV}_{\s_a}$. Let now $\G^*$ be the union of $\G$,
of the $\G_j^*$'s with $j>0$, of $\G_0^*\setm \overline{U^*V^*}_{\G_0^*}$
and of $\overline{U^*V^*}_{\s_a}$. This $\G^*$ is the union of
$k_0+1$ piecewise smooth Jordan curves.

Now let
\be m=\left[(1+\e)^7A_-n/NA_+\right]\label{mdef}\ee
and consider the polynomial
\be P_{n+mN}(z)=Q_n(z)T_N(z)^m\label{pdef**}\ee
on $\G^*$ with the $Q_n$ from (\ref{q1}), and let the measure $\mu^*$ be the measure
in (\ref{mu*}) on $\G^*$.
For the polynomials $P_{n+Nm}$ it was shown in
\cite[(3.18)-(3.20)]{Totikarc}  that
on $\G^*\setm  \left(\overline{UV}_{\G_0}\cup \overline{U^*V^*}_{\s_a}\right)$,
\be |P_{n+mN}(z)|\le C_1n^{1/2} e^{na^2-ma},\label{1}\ee
on $\overline{UV}_{\G_0}$
\be  |P_{n+mN}(z)|\le |Q_n(z)|,\label{2}\ee
and on $\overline{U^*V^*}_{\s_a}$
\be  |P_{n+mN}(z)|\le C_1 n^{1/2}\exp \left(n(1+\e)^4aA_-/|T_N'(z_0)|-ma\right),\label{3}\ee
where $C_1$ is a fixed constant. Here, by the choice of $m$ in (\ref{mdef}), and by (\ref{gk0}) and (\ref{gn}),
 the last exponent is at most
\[n\left(\frac{(1+\e)^5aA_-}{A_+N}-\frac{(1+\e)^6aA_-}{NA_+}\right)=-\e n\frac{(1+\e)^5aA_-}{NA_+}.\]

Fix $a$ so small that we have $a^2-aA_-/NA_+<0$. Then the
inequality $|T_N(z)|\le 1$ for $z\in \G^*$ and the
estimates (\ref{1})--(\ref{3})
yield
\[\l_{n+mN}(\mu^*,z_0)\le \int|P_{n+mN}|^2d\mu^*\le \int|Q_n|^2d\mu+O(n^{-\a-2}).\]
Hence, by (\ref{q1}),  for infinitely many $n$
\be (n+mN)^{\a+1}\l_{n+mN}(\mu^*,z_0)\le \left(\frac{n+mN}{n}\right)^{\a+1}\left(1-\d\right)\frac{w(z_0)L_\a}{(\pi \o_\G(z_0))^{\a+1}}+o(1).\label{lmn}\ee
Since (see \cite[(3.22)--(3.23)]{Totikarc})
\be \o_\G(z_0)=\frac{1}{2\pi}\left(\frac{\partial g_{\O}}{\partial{\bf n}_+}+\frac{\partial g_{\O}}{\partial{\bf n}_-}\right)=\frac{1}{2\pi}(A_++A_-)
\label{og0}\ee
and
\be  \o_{\G^*}(z_0)=\frac{1}{2\pi}
\frac{\partial g_{\O^*}(z_0)}{\partial {\bf n}_+}\le \frac{1}{2\pi}
\frac{\partial g_{\O}(z_0)}{\partial {\bf n}_+}= \frac{1}{2\pi} A_+,\label{og1}\ee
we have
\bean \left(\frac{n+mN}{n}\right)^{\a+1}
\left(1-\d\right)&&\hskip-15pt\frac{w(z_0)L_\a}{(\pi \o_\G(z_0))^{\a+1}}\le  \left(1+(1+\e)^7\frac{A_-}{A_+}\right)^{\a+1}\times\\
&&\hskip-1cm \left(1-\d\right)
\frac{w(z_0)L_\a}{(\pi \o_{\G^*}(z_0))^{\a+1}}
\left(\frac{A_+}{A_++A_-}\right)^{\a+1}\\
&&\le \left(1-\frac{\d}{2}\right)\frac{w(z_0)L_\a}{(\pi \o_{\G^*}(z_0))^{\a+1}}\eean
if $\e$ is sufficiently small. Therefore, (\ref{lmn}) implies
\[ \liminf_{n\to\i} (n+mN)^{\a+1}\l_{n+mN}(\mu^*,z_0)\le \left(1-\frac{\d}{2}\right)
\frac{w(z_0)L_\a}{(\pi \o_{\G^*}(z_0))^{\a+1}},\]
which is impossible according to Proposition \ref{thmainprop}.
This contradiction shows that (\ref{q1}) is impossible,
and so
\be \liminf_{n\to\i}n\l_n(\mu,z_0) \ge \frac{w(z_0)L_\a}{(\pi \o_\G(z_0))^{\a+1}}.
\label{veg4}\ee
follows.
\endproof

(\ref{veg2}) and (\ref{veg4}) prove Proposition \ref{thmainprop*}.

\sect{Proof of Theorem \ref{thmain}}
Let $\G$ be  as in the theorem, and let $\G=\cup_{k=0}^{k_0} \G^k$ be
the connected components of $\G$. Let $\O$ be the unbounded
connected component of $\ov\C\setm \G$.
We may assume that $z_0\in \G_0$. By assumption, $z_0$ lies
on a $C^2$-smooth arc $J$ of $\partial \O$, and there is an open set
$O$ such that $J=\G\cap O$. Let $\D_\d(z_0)$ be a small
disk about $z_0$ that lies in $O$ together with its closure.
Now there are two possibilities for $J$:
\begin{description}
\item[Type I] only one side of $J$ belongs to $\O$,
\item[Type II] both sides of $J$ belong to $\O$.
\end{description}
Type I occurs when $\G^0\setm \D_\d(z_0)$ is connected,
and Type II occurs when this is not the case.

Let $g_{\O}(z)$ be the Green's function for the domain
$\O$ with pole at infinity, which we assume
to be defined to be 0 outside $\O$.
The proof of Theorem \ref{thmain}
is based on  the following propositions.

\begin{prop}\label{propt1} If $J$ is of Type I, then
there is a sequence $\{\G_m\}$ of sets consisting of
disjoint $C^2$-smooth Jordan curves $\G_m^k$, $k=0,1,\ldots,k_0$,
such that with some positive sequence $\{\e_m\}$ tending to 0
we have
\begin{description}
\item[(i)] $z_0\in \G_m^0$ and $\G\cap \ov{\D_\d(z_0)}=
 \G_m\cap \ov{\D_\d(z_0)}$,
\item[(ii)]
\[\frac1{1+\e_m}\o_\G(z_0)\le \o_{\G_m} (z_0)\le (1+\e_m)\o_\G(z_0),\]
\item[(iii)]
\[\max _{x\in \G_m}g_{\O}(z)\le \e_m,\qquad
\max _{x\in \G}g_{\O_m}(z)\le \e_m.\]
\item[(iv)] The Hausdorff distance of the outer boundaries
of $\G$ and $\G_m$ tends to 0 as $m\to\i$.
\end{description}
\end{prop}
  Property {\bf (i)}
means that  in the $\d$-neighborhood of $z_0$
the sets $\G_m$ and $\G$ coincide.

\begin{prop}\label{propt2} If $J$ is of Type II, then
there is a sequence $\{\G_m\}$ of sets consisting of $\G_m^0:=J\cap \ov{\D_\d(z_0)}$
and of disjoint $C^2$  Jordan curves $\G_m^k$, $k=1,\ldots,k_0+2$,
lying in the component of $\G_m^0$ such that {\bf (i)}--{\bf (iv)} above hold.
\end{prop}

Pending the proofs of these propositions now we complete the proof of
Theorem \ref{thmain}. It follows from {\bf (i)} and {\bf(iv)}
that there is a compact set $K$ that contains $\G$ and all $\G_m$
such that $z_0$ lies on the outer boundary of $K$, and
in a neighborhood of $z_0$ the outer boundary of $K$ and $\G$
are the same. In particular, there is a circle in the unbounded
component of $\ov\C\setm K$ that contains $z_0$ on its boundary,
so we can apply Proposition \ref{fdp} to $K$ and $z_0$.

Fix an $m$ and consider the set $\G_m$ either from
Proposition \ref{propt1} if $J$ is of Type I or from Proposition \ref{propt2}
if $J$ is of Type II. We define the measure
\[\mu_m(z)=w(z)|z-z_0|^\a ds_{\G_m}(z),\]
where $w$ is a continuous and positive extension of the original
$w$ (that existed on $J$) from $J\cap \ov{\D_\d(z_0)}$
to $\G_m$. It follows from the Erd\H{o}s-Tur\'an criterion
\cite[Theorem 4.1.1]{StahlTotik}
that this $\mu_m$ is in the ${\bf Reg}$ class.

For positive integer $n$
let $P_n$ be the extremal polynomial
of degree $n$ for $\l_n(\mu,z)$.
Consider the polynomial $S_{4n\e_m/c_2\d^2,z_0,K}(z)$ from Proposition \ref{fdp}
with $\g=2$ (here $c_2$ is the constant from
Proposition \ref{fdp}), and form the product $Q_n(z)=P_n(z) S_{4n\e_m/c_2\d^2,z_0,K}(z)$.
This is a polynomial of degree at most $n(1+4\e_m/c_2\d^2)$ which takes
the value 1 at $z_0$. On $\G_m \cap \ov{\D_\d(z_0)}=\G \cap \ov{\D_\d(z_0)}$ we have
\be \int_{\G_m \cap \ov{\D_\d(z_0)}}|Q_n(z)|^2\le
\int_{\G \cap \ov{\D_\d(z_0)}}|P_n(z)|^2\le \l_n(\mu,z_0).
\label{gnet}\ee

Since the $L^2(\mu)$-norms of $\{P_n\}$ are
bounded, it follows from
$\mu\in {\bf Reg}$ that there is an $n_m$ such that if
$n\ge n_m$, then we have
\[\|P_n\|_\G\le e^{\e_m n}.\]
Then, by the
 Bernstein-Walsh lemma (Lemma \ref{BW}) and by property
 {\bf (iii)}, we have
 for all $z\in \G_m$
\[ |P_n(z)|\le \|P_n\|_\G e^{ng_\O(z)}
\le e^{2n\e_m}.\]
Therefore, (\ref{kineq}) and $\G_m\subseteq K$ imply that for
$z\in \G_m\setm \ov{\D_\d(z_0)}$
\[|Q_n(z)|\le \exp(2n\e_m-[4n\e_m/c_2\d^2]c_2\d^2)<e^{-n\e_m}\]
if $n$ is sufficiently large. As a consequence, the integral
of $Q_n$ over $\G_m\setm \ov{\D_\d(z_0)}$ is exponentially
small in $n$, which, combined with (\ref{gnet}), yields that
\[\l_{n(1+4\e_m/c_2\d^2)}(\mu_m,z_0)\le \l_n(\mu,z_0)+o(n^{-(1+\a)}).\]
Multiply here both sides by $n(1+4\e_m/c_2\d^2)^{1+\a}$ and
let $n$ tend to infinity. If we apply that Theorem \ref{thmain} has already
been proven for $\G_m$ and for the measure
$\mu_m$ (see Proposition \ref{thmainprop*}),
we can conclude (use also (\ref{liminf}))
\[\liminf_{n\to\i}n^{\a+1}\l_{n}(\mu,z_0)\ge \frac{1}{1+4\e_m/c_2\d^2}\frac{w(z_0)}{(\pi\o_{\G_m}(z_0))^{\a+1}}L_\a\]
(with the $L_\a$ from (\ref{lalpha})), and an application
of property {\bf (ii)} yields then
\[\liminf_{n\to\i}n^{\a+1}\l_{n}(\mu,z_0)\ge
\frac{1}{(1+\e_m)^{|\a|+1}(1+4\e_m/c_2\d^2)}\frac{w(z_0)}{(\pi\o_\G(z_0))^{\a+1}}L_\a.\]

If we reverse the roles of $\G$ and $\G_m$ in this argument, then we
can similarly conclude
\[\limsup_{n\to\i}n^{\a+1}\l_{n}(\mu,z_0)\le
(1+\e_m)^{|\a|+1}(1+4\e_m/c_2\d^2)\frac{w(z_0)}{(\pi \o_\G(z_0))^2}L_\a.\]
Finally, in these last two relations we can let $m\to\i$, and as
$\e_m\to 0$, the limit in Theorem \ref{thmain} follows.

Thus, it is left to prove
Propositions \ref{propt1} and \ref{propt2}.

\subsection{Proof of Proposition \ref{propt1}}
Both in this proof and in the next one we shall
use that if $\O_1\subset \O_2$ (say both with
a smooth boundary),  and $z\in \O_1$, then
$g_{\O_1}(z)\le g_{\O_2}(z)$. As a consequence,
if $z$ is a common point on their boundaries, then
the normal derivative of $g_{\O_1}$ (the normal pointing
inside $\O_1$) is not larger than the same normal derivative
of $g_{\O_2}$ (because both Green's functions vanish
on the common boundary). Since, modulo a factor $1/2\pi$, the
normal derivatives yield
the equilibrium densities (see formulae (\ref{norm1}) and
(\ref{norm2}) below), it also follows that
if $\G_1\subset \G_2$, then on (an arc of) $\G_1$ the equilibrium
density $\o_{\G_2}$ is at most as least as large as
the equilibrium density $\o_{\G_1}$ (see
also \cite[Theorem IV.1.6(e)]{SaffTotik},
according to which the equilibrium measure
for $\G_1$ is the balayage onto $\G_1$
of the equilibrium
measure of $\G_2$).

Choose, for each $m$ and $1\le k\le k_0$, $C^2$-smooth Jordan curves $\G_m^k$
so that they lie in $\O$ and are of distance $<1/m$ from $\G^k$.
For $k=0$ the choice is somewhat different: let $\G_m^0$ be
a $C^2$ Jordan curve that lies in $\ov \O$, its distance from
$\G^0$ is smaller than $1/m$, $J\cap \ov {\D_\d(z_0)}\subset
\G_m^0$, and $\G_m^0\setm J$ lies in $\O$, see Figure \ref{fig1}.
We can select these so that the outer domains $\O_m$ of $\G_m$
are increasing with $m$.
From this construction it is clear that {\bf (i)} and {\bf (iv)}
are true.
\begin{figure}
\centering
\includegraphics[scale=0.6]{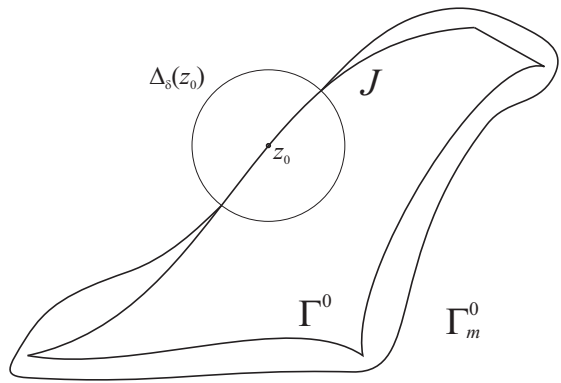}
\caption{\protect\label{fig1}The arc $J$ and the selection of $\G_m^0$}
\end{figure}
Now $\ov\C\setm \O_m$ (the so called
polynomial convex hull of $\G_m$) is a shrinking sequence of
compact sets, the intersection of which is $\ov\C\setm \O$.
Therefore, if $\c$ denotes the logarithmic capacity, then we have
(see \cite[Theorem 5.1.3]{Ransford}) $\c(\ov\C\setm \O_m)\to \c(\ov\C\setm \O)$.
Since
$\{g_{\O}(z)-g_{\O_m}(z)\}$ is a decreasing sequence of positive harmonic functions
(more precisely, this sequence starting from the term $g_{\O}(z)-g_{\O_l}(z)$ is
harmonic in $\O_l$)
for which (see \cite[Theorem 5.2.1]{Ransford})
\[g_{\O}(\i)-g_{\O_m}(\i)=\log\frac1{\c(\ov\C\setm \O)}-\frac{1}{\c(\ov\C\setm \O_m)}\to 0,\]
we obtain from Harnack's theorem (\cite[Theorem 1.3.9]{Ransford}) that
$g_{\O}(z)-g_{\O_m}(z)\to 0$ locally uniformly on compact subsets of
$\O$. This, and the fact that this sequence is defined in $\O\cap \D_\d(z_0)$
and has boundary values identically 0 on $\partial \O\cap \D_\d(z_0)$,
then implies (see e.g. \cite[Lemma 7.1]{NT}) the following: if ${\bf n}$ denotes the normal to $z_0$ in the direction
of $\O$ then, as $m\to\i$,
\[ \frac{\partial g_{\O_m}(z_0)}{\partial {\bf n}}\to \frac{\partial g_{\O}(z_0)}{\partial {\bf n}}.\]
But in the Type I situation we have  (see \cite[II.(4.1)]{Nevanlinna} combined
with \cite[Theorem 4.3.14]{Ransford} or \cite[Theorem IV.2.3]{SaffTotik} and \cite[(I.4.8)]{SaffTotik})
\be \o_\G(z_0)=\frac1{2\pi}  \frac{\partial g_{\O}(z_0)}{\partial {\bf n}},\label{norm1}\ee
and a similar formula is true for $\o_{\G_m}$, hence
\[\o_{\G_m}(z_0)\to \o_\G(z_0),\qquad m\to\i.\]
This takes care of {\bf (ii)}.

Finally, we use the following statement from \cite[Theorem 7.1]{Totiktrans}:
\begin{lem} \label{unif}  Let $S$ be a continuum. Then
the Green's function $g_{\ov\C \setm S}(z,\i)$
is uniformly H\"older $1/2$ continuous
 on $S$, i.e. if
$z_0\in \O$, then
\be g_{\ov\C \setm S}(z_0,\i)\le C{\rm dist}(z_0,S)^{1/2}.\label{m1}\ee
Furthermore, here $C$ can be chosen to depend only on the diameter of $S$.
\end{lem}

If we apply this with $S=\G^k$, $k=0,\ldots,k_0$ and
use that $g_{\O_m}(z)\le g_{\O_m^k}(z)$ for each $k$
(where, of course,
$\O_m^k$ is the unbounded component of $\ov\C\setm \G_m^k$),
then we can conclude the first inequality in {\bf (iii)}.
In this case (i.e. when $J$ is of Type I), the second inequality in
{\bf (iii)} is trivial, since, by the construction,
 $g_{\O_m}$ is identically 0 on $\G$.\endproof

\subsection{Proof of Proposition \ref{propt2}} For an $m$
let $J_{1,m}$ resp. $J_{2,m}$ be
the two open subarcs of $J$ of diameter $1/m$ that lie outside $\D_\d(z_0)$, but which
have one endpoint in $\ov{\D_\d(z_0)}$ (see Figure \ref{fig2}) (for large $m$ these exist).
\begin{figure}
\centering
\includegraphics[scale=0.5]{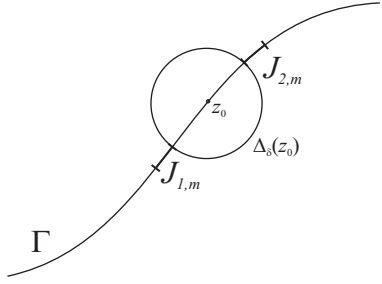}
\caption{\protect\label{fig2}The arcs $J_{1,m}$ and $J_{2,m}$}
\end{figure}

Remove now $J_{1,m}$ and $J_{2,m}$ from $\G$. Since we are in the Type II situation,
 after this removal
the unbounded component of the complement of $\G^0\setm ({J_{1,m}\cup J_{2,m}})$
is $\O\cup J_{1,m}\cup J_{2,m}$, and  $\G^0\setm (J_{1,m}\cup J_{2,m})$ splits into three
connected components, one of them being $J\cap \ov{\D_\d(z_0)}$. Let
$\G^{0,1},\G^{0,2}$ be the other two components of
$\G^0\setm ({J_{1,m}\cup J_{2,m}})$. As $m\to\i$ we have
$\c(\ov\C\setm (\O \cup {J_{1,m}\cup J_{2,m}}))\to \c(\ov\C\setm \O)$,
and since now the domains $\O \cup {J_{1,m}\cup J_{2,m}}$ are shrinking,
we can conclude from Harnack's theorem as before that
$g_{\O}(z)-g_{\O_m}(z)\to 0$ locally uniformly on compact subsets of
$\O$. This implies again that if ${\bf n}_\pm$ are the
two normals to $\G$ at $z_0$ (note that now both point inside $\O$),
then
\[\frac{\partial g_{\O \cup {J_{1,m}\cup J_{2,m}}}(z_0)}{\partial {\bf n\pm}}\to
\frac{\partial g_{\O}(z_0)}{\partial {\bf n}}\]
as $m\to\i$. Since now  (see \cite[II.(4.1)]{Nevanlinna} or \cite[Theorem IV.2.3]{SaffTotik} and \cite[(I.4.8)]{SaffTotik})
\be \o_\G(z_0)=\frac1{2\pi} \left ( \frac{\partial g_{\O}(z_0)}{\partial {\bf n}_+}
+ \frac{\partial g_{\O}(z_0)}{\partial {\bf n}_-}\right),\label{norm2}\ee
we can conclude again
that
\be 0\le \o_{\G \setm (J_{1,m}\cup J_{2,m})}(z_0)- \o_\G(z_0)< \e_m
\label{ryt}\ee
with some $\e_m>0$ that tends to 0 as $m\to\i$. By selecting a
somewhat larger $\e_m$ we may also assume
\be g_{\O\cup{J_{1,m}\cup J_{2,m}}}(z)<\e_m,\qquad z\in J_{1,m}\cup J_{2,m}\label{asd*}\ee
(apply Lemma \ref{unif} to $S=\G\cap\ov{\D_\d(z_0)}$ and
use that
$g_{\O\cup{J_{1,m}\cup J_{2,m}}}(z)\le g_{\ov\C\setm (\G\cap\ov{\D_\d(z_0)})}(z)$).

For the continua $\G^{0,1},\G^{0,2},\G_1,\G_2,\ldots,\G_{k_0}$
and for a small $0<\t<1/m$
select $C^2$-smooth Jordan curves $\g^{0,1},\g^{0,2},\g_1,\g_2,\ldots,\g_{k_0}$
that lie in $\O \cup {J_{1,m}\cup J_{2,m}}$ and are of distance $<\t$ from
the corresponding continuum. Let $\G_{m,\t}$ be the union
of $J\cap \ov{\D_\d(z_0)}$ and of these last chosen Jordan curves.
Then $\G_{m,\t}$ consists (for small $\t$) of $k_0+2$ Jordan curves and one
Jordan arc (namely $J\cap \ov{\D_\d(z_0)}$), all of them $C^2$-smooth.
According to the proof of Proposition \ref{propt1} we have
\[\o_{\G_{m,\t}}(z_0)\to \o_{\G\setm (J_{1,m}\cup J_{2,m})}(z_0)\]
as $\t\to 0$, therefore, for sufficiently small $\t$, we have (see (\ref{ryt}))
\[-\e_m<\o_{\G_{m,\t}}(z_0)- \o_\G(z_0)< \e_m.\]

Thus, if $\t$ is sufficiently small, we have properties {\bf (i), (ii)}
and {\bf (iv)} in the proposition for $\G_m=\G_{m,\t}$. The first inequality in
{\bf (iii)} follows exactly as at the end of the proof of Proposition
\ref{propt1}. Finally, the second inequality in {\bf (iii)} follows from
(\ref{asd*}) because
\[g_{\O_{m,\t}}(z)\le g_{\O\cup{J_{1,m}\cup J_{2,m}}}(z),\]
(where $\O_{m,\t}$ is the unbounded component of $\ov\C\setm \G_{m,\t}$)
and $g_{\O_{m,\t}}(z)=0$ if $z\in \G$ unless $z\in J_{1,m}\cup J_{2,m}$.

These show that for sufficiently small $\t$ we can
select $\G_m$ in Proposition \ref{propt2}
as $\G_{m,\t}$.\endproof

\sect{Proof of Theorem \ref{thmainarc}}
Let $\G$ be  as in Theorem \ref{thmainarc}, and let $\G=\cup_{k=0}^{k_0} \G_k$ be
the connected components of $\G$, $\G_0$ being the one that contains
$z_0$. We may assume that
$z_0=0$. Set
\[\tilde \G=\{z\sep z^2\in \G\}, \qquad \tilde \G_k=\{z\sep z^2\in \G_k\}.\]
Every $\tilde \G_k$ is the union of two disjoint continua: $\tilde \G_k=\G_k^{+}\cup \tilde \G_k^{-}$,
where $\tilde \G_k^-=-\tilde \G_k^+$. Set $\tilde \G^\pm=\cup_k \tilde \G_k^{\pm}$.
All the $\tilde \G_k^{\pm}$ are disjoint, except when $k=0$:
then 0 is a common point of $\G_0^{\pm}$, but except for that point,
$\tilde \G_0^{+}$ and $\tilde \G_0^{-}$ are again disjoint.
In general, we shall use the notation $\tilde H$ for the set of points
$z$ for which $z^2$ belongs to $H$, and if $H$ is a continuum, then represent
$\tilde H$ as the union of two continua $\tilde H^+\cup \tilde H^-$,
where $\tilde H^-=-\tilde H^+$, and $\tilde H^-$ and $\tilde H^+$
are disjoint except perhaps for the point 0 if 0 belongs to $H$.

Now $\tilde \G_0^+\cup \tilde \G_0^-$ is connected, and if $J$ is the
$C^2$-smooth arc of $\G$ with
one endpoint at $z_0=0$, then $\tilde J$ is a $C^2$-smooth arc that lies on the
outer boundary of $\tilde \G$, and  $\tilde J$  contains 0 in its
(one-dimensional) interior.
Thus, $\tilde \G$ and $z_0=0$ satisfy the assumptions in Theorem \ref{thmain}.

For a measure $\mu$ defined on $\G$ let $\tilde \mu$ be the measure
$d\tilde \mu(z)=\frac12 d\mu(z^2)$, i.e. if, say, $E\subset \tilde \G^+$
is a Borel set
and $E^2=\{z^2\sep z\in E\}$, then
\[\tilde \mu(E)=\frac12 \mu(E^2),\]
and a similar formula holds for $E\subset \G^-$. So $\tilde \mu$ is
an even measure, which has the same total mass as $\mu$ has.

Let $\nu_\G$ be the equilibrium measure of $\G$. We claim that
$\nu_{\tilde \G}=\widetilde {\nu_\G}.$ Indeed, for any $z\in \tilde \G$
we have
\bean \int \log|z-t|d\widetilde {\nu_\G}(t)&=&\int_{\tilde \G^+}(\log |z-t|+\log|z+t|)
d\widetilde {\nu_\G}(t)\\
&=&\frac12 \int_\G \log|z^2-t^2|d\nu_\G(t^2)\\
&=&\frac12\int \log|z^2-u|d\nu_\G(u)={\rm const}\eean
because the equilibrium potential of $\nu_\G$ is constant
on $\G$  by Frostman's theorem (see \cite[Theorem 3.3.4]{Ransford}),
and $z^2\in \G$. Since the equilibrium measure $\nu_{\tilde \G}$
is characterized (among all probability measures on $\tilde \G$)
by the fact that its logarithmic potential
is constant on
the given set, we can conclude that $\widetilde {\nu_\G}$ is, indeed,
the equilibrium measure of $\tilde \G$ (here we use that all
the sets which we are considering are the unions of finitely
many continua, hence the equilibrium potentials for them
are continuous everywhere).

Let $\g(t)$ be a parametrization of $\tilde J^+$ with $\g(0)=0$. Then $\g(t)^2$
is a parametrization of $J$, and the two corresponding arc measures are
$|\g'(t)|dt$ and $|(\g(t)^2)'|dt=2|\g(t)||\g'(t)|dt$, resp.
Therefore, since the $\nu_{\tilde \G}$-measure
of an arc $\{\g(t)\sep t_1\le t\le t_2\}$ is the same
as half of the $\nu_\G$-measure of the
arc $\{\g(t)^2\sep t_1\le t\le t_2\}$, we have
\[\int_{t_1}^{t_2} \o_{\tilde \G}(\g(t))|\g'(t)|dt
= \frac12\int_{t_1}^{t_2} \o_\G(\g(t)^2)2|\g(t)||\g'(t)|dt,\]
from which
\[\o_{\tilde \G}(\g(t))=\o_\G(\g(t)^2)|\g(t)|, \qquad t\in \tilde J^+,\]
follows (recall, that on both sides the $\o$ is the
equilibrium density with respect to the corresponding
arc measure). A similar formula
holds on $\tilde J^-$. But $\o_{\tilde \G}(z)$ is continuous
and positive at $0$ (see e.g. \cite[Proposition 2.2]{Totikarc}), therefore the preceding
formula shows that $\o_\G(z)$ behaves
around 0 as $\o_{\tilde \G}(0)/\sqrt {|z|}$, and
we have (see (\ref{mgdef}) for the definition of $M(\G,0)$)
\be M(\G,0)=\lim_{z\to0} \sqrt{|z|}\o_\G(z)=\o_{\tilde \G}(0).\label{mgo}\ee

Now the same argument that was used in the proof of Proposition
\ref{model_2} (see in particular (\ref{mod01})) shows that
\be \l_{2n}(\tilde \mu,0)=\l_n(\mu,0).\label{ketl}\ee
$\mu$ was assumed to be of the form $w(z)|z|^\a ds_J(z)$
on $J$, hence, as before,
\[\int_{t_1}^{t_2} d\tilde \mu(t)
= \frac12\int_{t_1}^{t_2} w(\g(t)^2)|\g(t)^2|^\a2|\g(t)||\g'(t)|dt,\]
and since here $|\g'(t)|dt$ is the arc measure on $\tilde J^+$,
we can conclude that on $\tilde J^+$ the measure
$\tilde \mu$ has the form
$d\tilde \mu(z)=w(z^2)|z|^{2\a+1}ds_{\tilde J}(z)$, and the same representation
holds on $\tilde J^{-}$. Therefore, Theorem \ref{thmain}
can be applied to the set $\tilde \G$, to the measure $\tilde \mu$
and to the point $z_0=0$, the only change is that now $\a$ has
to be replaced by $2\a+1$ when dealing with the measure
$\tilde \mu$. Now we obtain from
(\ref{ketl})
\[\lim_{n\to\i} (2n)^{2\a+2}\l_{2n}(\tilde \mu,0)
=\lim_{n\to\i} (2n)^{2\a+2}\l_n(\mu,0),\]
and since, according to Theorem \ref{thmain},  the limit on the left
is
\[2^{2\alpha + 2} \Gamma\Big( \frac{2\alpha + 2}{2} \Big) \Gamma\Big( \frac{2\alpha + 4}{2} \Big)
\frac{w(0)}{(\pi \o_{\tilde \G}(0))^{2\a+2}},\]
we obtain
\[\lim_{n\to\i} n^{2\a+2}\l_n(\mu,0)=\Gamma\Big( \alpha + 1 \Big) \Gamma\Big( \alpha + 2 \Big)
\frac{w(0)}{(\pi \o_{\tilde \G}(0))^{2\a+2}},\]
which, in view of (\ref{mgo}), is the same as
(\ref{arceq}) in Theorem \ref{thmainarc}.\endproof

Vilmos Totik

Bolyai Institute

MTA-SZTE Analysis and Stochastics Research Group

University of Szeged

Szeged

Aradi v. tere 1, 6720, Hungary

\smallskip
and
\smallskip

Department of Mathematics and Statistics

University of South Florida

4202 E. Fowler Ave, CMC342

Tampa, FL 33620-5700, USA
\medskip

{\it totik@mail.usf.edu}

\bigskip

Tivadar Danka

Potential Analysis Research Group, ERC Advanced Grant No. 267055

Bolyai Institute

University of Szeged

Szeged

Aradi v. tere 1, 6720, Hungary

{\it tivadar.danka@math.u-szeged.hu}
\end{document}